\documentclass[reqno,12pt,letterpaper]{amsart}
\usepackage[latin1]{inputenc}
\usepackage[T1]{fontenc}
\usepackage[english]{babel}
\usepackage{amsmath}
\usepackage{amsthm}
\usepackage{amsfonts}
\usepackage{amssymb}
\usepackage{mathtools}
\mathtoolsset{showonlyrefs}
\usepackage{mathrsfs}
\usepackage{cool}

\usepackage{microtype}
\usepackage{enumerate}
\usepackage{lmodern}
\newtheorem{thm}{Theorem}[section]
\newtheorem{prop}[thm]{Proposition}
\newtheorem{lem}[thm]{Lemma}
\newtheorem{cor}[thm]{Corollary}

\theoremstyle{definition}

\theoremstyle{remark}
\newtheorem{rem}[thm]{Remark}

\numberwithin{equation}{section}

\newcommand{\ang}[1]{{\left\langle{#1}\right\rangle}}

\renewcommand{\Re}{\operatorname{Re}}

\newcommand{\pa}{{\partial}}
\newcommand{\ep}{{\varepsilon}}

\newcommand{\RR}{\mathbb{R}}

\newcommand{\NN}{\mathbb{N}}

\newcommand{\CC}{\mathbb{C}}

\newcommand{\br}[1]{\left(#1\right)}

\newcommand{\ol}{\overline}

\DeclareMathOperator{\WF}{WF}
\DeclareMathOperator{\Op}{Op}

\newcommand{\injto}{\hookrightarrow}
\DeclareMathOperator{\Id}{Id} 
\DeclareMathOperator{\Span}{Span}

\DeclareMathOperator{\End}{End} 

\newcommand{\SA}{\mathcal A}

\newcommand{\SC}{\mathcal C}
\newcommand{\SD}{\mathcal D}
\newcommand{\SE}{\mathcal E}

\newcommand{\SN}{\mathcal N}

\newcommand{\ST}{\mathcal T}
\newcommand{\SX}{\mathcal X}
\newcommand{\SY}{\mathcal Y}

\newcommand{\SCE}{\mathscr E}

\newcommand{\SCG}{\mathscr G}
\newcommand{\SCH}{\mathscr H}

\usepackage[dvipsnames,svgnames]{xcolor}
\usepackage{url}
\colorlet{linkcolour}{black}
\colorlet{urlcolour}{blue}
\usepackage[colorlinks=true]{hyperref}
\hypersetup{colorlinks=true,
	linkcolor=DarkRed,
	citecolor=PineGreen,
	urlcolor=Magenta}

 \usepackage{tikz}
\usetikzlibrary {patterns}
\usetikzlibrary {arrows.meta} 

\newcommand{\pd}{\partial}

\newcommand{\spt}{\mathrm{spt}}
\newcommand{\cl}{\nabla^*\nabla }

\begin{document}
\title{The Lorentzian Calder\'{o}n problem on vector bundles}

\author{Se\'{a}n Gomes}
\address{Department of Mathematics and Statistics,
University of Helsinki,
P.O 68, 00014, University of Helsinki}
\email{sean.p.gomes@gmail.com}

\author{Lauri Oksanen}
\address{Department of Mathematics and Statistics,
University of Helsinki,
P.O 68, 00014, University of Helsinki}
\email{lauri.oksanen@helsinki.fi}

\maketitle

\begin{abstract}
	In this paper we study a Lorentzian version of the Calder\'{o}n problem, which is concerned with the determination of a connection and potential on a Hermitian vector bundle over a Lorentzian manifold from the Dirichlet-to-Neumann map of the associated connection wave operator. For a class of Lorentzian manifolds satisfying a curvature bound, including perturbations of Minkowski space over strictly convex domains, the connection and potential is shown to be uniquely determined up to the natural gauge transformations of the problem. The proof is based on ideas from the earlier works \cite{LO1},\cite{LO2} of the second author in the scalar setting.
\end{abstract}

\section{Introduction}
The purpose of this paper is to solve an inverse problem for the connection wave operator $P=\nabla^*\nabla+V$, acting on sections of a Hermitian vector bundle $E$ equipped with compatible connection $\nabla$ and Hermitian potential $V$, over a fixed Lorentzian manifold $(M,g)$. Boundary measurements, in the form of the Dirichlet-to-Neumann map associated to $P$, are shown to uniquely determine the connection $\nabla$ and the potential $V$ up to a natural gauge invariance in a recovery domain $\SD\subset M$ specified by the causal structure of $M$.
Our geometric setup is as follows.
\label{sec:intro}
Let
\begin{equation}
	\label{eq:productform}
	M=[-T,T]\times M_0
\end{equation}
be a $(1+n)$-dimensional Lorentzian manifold, where $T > 0$ and $M_0$ is compact and connected, with smooth boundary. We assume $M$ is equipped with the metric 
\begin{equation}\label{eq:metricform}
	g=c(t,x)(-dt^2 +g_0(t,x,dx)).
\end{equation}
were $g_0$ is a smooth family of Riemannian metrics parametrised by $t$ and the conformal factor $c(t,x)$ is smooth and positive.
Let $E$ be a Hermitian bundle of rank $N$ over $M$, equipped with a connection $\nabla$ that is compatible with the Hermitian structure, in the sense that
\begin{equation}\label{eq:conn.compat}
	X \ang{u,v}_E=\langle \nabla_X u,v\rangle_E+\ang{u,\nabla_X v}_E
\end{equation}
pointwise for $X\in\SC^\infty(M;TM)$ and $u,v\in\SC^\infty(M;E)$. For $X=\partial_t$, we use the notation $\nabla_t:=\nabla_{\partial_t}$.

Furthermore we assume that $E$ is equipped with the structure of a $G$-bundle for some Lie subgroup $G\subseteq U(\CC,N)$, that is a maximal collection of local trivialisations $\phi_\alpha:E|_{U_\alpha}\to U_\alpha\times \CC^N$ such that the transition maps satisfy $$\phi_\alpha^{-1}\circ\phi_\beta|_{U_\alpha\cap U_\beta}\in \SC^\infty(U_\alpha\cap U_\beta;(U_\alpha\cap U_\beta) \times G). $$
The maps $\phi_\alpha$ are said to be a \emph{$G$-charts} for $E$, together comprising a \emph{$G$-atlas} for $E$, and 
the pullbacks of the standard frames for the trivial bundles $U_\alpha\times \CC^N\to U_\alpha$ under $\phi_\alpha$ are said to be the  \emph{$G$-frames} for $E$. Sections of $\mathrm{End}(E)$ are said to be \emph{$G$-sections} if $(\phi^{-1})^* A\phi^*\in \SC^\infty(U;U\times G) $ for any $G$-chart $\phi:E|_U\to U\times \CC^N$. We denote the set of $G$-sections by $\SC^\infty(M;G(E))$.

We assume that $\nabla$ is compatible with the $G$-bundle structure, in the sense that parallel transport with respect to $\nabla$ preserves the collection of $G$-frames. Compatibility with the Hermitian structure corresponds to the case $G=U(\CC,N)$ with the $G$-bundle structure of $E$ furnished by the local trivialisations induced by arbitrary local orthonormal frames.

A choice of $G$-char $\phi:E|_U\to U\times \CC^N$ over $U\subseteq M$ allows us to write any compatible connection in the form
\begin{equation}
\label{eq:B.def.0}
\nabla=\phi^*(d+B)(\phi^{-1})^*\quad (B\in \SC^\infty(U;\mathfrak{g}\otimes T^*U))
\end{equation}
where $d$ denotes the componentwise exterior derivative and $\mathfrak{g}$ is the Lie algebra of $G$. For $G=U(\CC,N)$, we have $\mathfrak{g}=\mathfrak{u}(\CC,N)$, the Lie algebra of skew-Hermitian matrices.

Let $V\in \SC^\infty(M;\End(E))$ be Hermitian, in the sense that
\begin{equation}\label{eq:potential.compat}
	\ang{Vu,v}_E=\ang{u,Vv}_E
\end{equation}
pointwise for $u,v\in\SC^\infty(M;E)$. Such $V$ will be referred to as \emph{potentials}.

There exist natural conjugate symmetric, non-degenerate sesquilinear forms on $\SC_c^\infty(\mathrm{int}(M);E)$ and $ \SC_c^\infty(\mathrm{int}(M);E\otimes T^*M)$ given by
\begin{equation}\label{eq:L2.def}
	\br{u,v}_{E}:=\int_{M} \ang{u,v}_E \, dV_{g}(x)
\end{equation}
and 
\begin{equation}\label{eq:L2.def.forms}
	\br{u\otimes \alpha,v\otimes \beta}_{E\otimes T^* M}:=\int_{M} \ang{u,v}_E \ang{\alpha,\beta}_g \, dV_{g}(x)
\end{equation}
for $u,v\in \SC_c^\infty(\mathrm{int}(M);E)$ and $\alpha,\beta\in\SC_c^\infty(\mathrm{int}(M); T^*M)$, where $dV_g$ is the Riemannian volume density induced by $g$ and \eqref{eq:L2.def.forms} is extended by linearity.
We define a formal adjoint $\nabla^*$ by 
\begin{equation}\label{eq:duality}
	\br{\nabla u,\omega}_{E\otimes T^* M}=\br{u,\nabla^* \omega}_{E}.
\end{equation}
The \emph{connection wave operator} induced by $\nabla$ is then defined by 
\begin{equation}\label{eq:box.def.0}
	\Box=\nabla^*\nabla.
\end{equation} 
More generally, we consider operators of the form 
\begin{equation}\label{eq:box.def}
	P =\Box+V
\end{equation}
for potentials $V\in\SC^\infty(M;\End(E))$ satisfying \eqref{eq:potential.compat}.

The \emph{connection wave equation} on $M$ is then given by
\begin{align}
	P u&=0\\
	u&= f \textrm{ on $\Sigma=(-T,T)\times \pd M_0$}\label{eq:wave.eq}\\
	(u,\nabla_t u)&=0 \textrm{ on $\{-T\}\times M_0$}.
\end{align}
The \emph{Dirichlet-to-Neumann map} $\Lambda:H^1_0(\Sigma;E)\to L^2(\Sigma;E)$ for a connection $\nabla$ and a potential $V$ is defined by
by 
\begin{equation}\label{eq:DtN}
	\Lambda f:=\nabla_\nu u|_{\Sigma}
\end{equation}
where $\nu$ is the outward pointing unit normal vector field on $\Sigma$, and $u$ is the unique solution to \eqref{eq:wave.eq}, with well-posedness of \eqref{eq:wave.eq} and regularity of \eqref{eq:DtN} following from Proposition \ref{prop:forwardsmooth}.

The inverse problem we consider (henceforth referred to as the \emph{Lorentzian Calder\'{o}n problem on Hermitian vector bundles}) is that of obtaining injectivity of the map
 $$(\nabla,V)\mapsto \Lambda_{\nabla,V}$$ 
up to two natural obstructions.

The first obstruction is the presence of a gauge invariance for \eqref{eq:wave.eq}. In particular, if $A\in\mathcal{C}^\infty(M;G(E))$ with $A|_\Sigma=\mathrm{Id}$ then it follows immediately that $$\Lambda_{A^*\nabla A,A^*VA}=\Lambda_{\nabla,V}.$$
The second obstruction is that of finite speed of propagation for $P$. Solutions to \eqref{eq:wave.eq} vanish in the subset of $M$ consisting of points which are not in the causal future of any point of $\Sigma$, and so we cannot hope to obtain information about $\nabla$ or $V$ in this region from knowledge of $\Lambda$. 

Our main result is that under suitable geometric hypotheses (essentially those in \cite{LO2}), coinciding Dirichlet-to-Neumann maps for $(\nabla_j,V_j)$ with $j=1,2$ implies gauge equivalence of $(\nabla_j,V_j)$ in a suitable recovery domain $\SD\subset M$ which is specified using the causal structure of $M$. We shall state the hypotheses of the main theorem now, and recall the definition of the causal relation in Section \ref{sec:lorentzian.geom}.

We denote the causal future and past of $p\in M$ by 
$$J^{+}(p):=\{q\in M:q \geq p\}$$
$$J^{-}(p):=\{q\in M: q \leq p\}$$
and introduce the notation
\begin{equation}\label{eq:exterior.nullcone.def}
	\SE_p= M \setminus (J^-(p)\cup J^+(p))
\end{equation} 
for the exterior of the double null cone. We use $R$ to denote the Riemann curvature tensor and we make the following assumptions.
\begin{enumerate}
	\item[(H1)] For any point $p\in M$, any spacelike vector $v\in T_p M$, and any null vector $N\in T_p M$ with $g(v,N)=0$, we have 
	\begin{equation}\label{eq:curvature}
		g(R(N,v)v,N)\leq 0
	\end{equation}
	\item[(H2)] For any null geodesic $\gamma$ between two points $p$ and $q$, $\gamma$ is the unique causal path between $p$ and $q$. For any $p\in M$, $\exp_p$ is a diffeomorphism from the subset of spacelike vectors (in its maximal domain of definition) onto $\SE_p$.
	\item[(H3)] There exists $T_0\in (-T,T)$ and $p_0\in \mathrm{int}(M)$ such that $\SCE_{p_0}\cap \pd M\subset \Gamma=(-T,T_0)\times \pd M_0\subset \Sigma$.
	\begin{figure}[h]
	\begin{center}
	\begin{tikzpicture}[scale=0.9]
	\draw (-1,-2)--(1,-2);
	\draw (-1,2)--(1,2);
	\draw (0,-2) node[below] {$M_0$};
	\draw (1,-2)--(1,2);
	\draw (-1,-2)--(-1,2);
	\draw (1,-2) node[right] {$t=-T$};
	\draw (1,2) node[right] {$t=T$};
	\draw (1,0) node[right] {$\Sigma$};

	\draw[dashed] (-1,-1)--(1,-2);
	\draw[dashed] (-1,-2)--(1,-1);
	\filldraw (0,-1.5) circle [radius=1pt];
	\draw (0,-1.5) node[above] {$p_0$};
	\draw (1,-1) node[right] {$t=T_0$};
	\draw[very thick,red] (1,-2)--(1,-1);
	\draw[very thick,red] (-1,-2)--(-1,-1);
	\draw (1,-1.5) node[right] {$\color{red}\Gamma$};

	\draw[dashed,pattern color=magenta,pattern=north east lines] (0,-1.5)--(1,-2)--(1,-1)--cycle;
		\draw[dashed,pattern color=magenta,pattern=north east lines] (0,-1.5)--(-1,-2)--(-1,-1)--cycle;
	\draw[->] (-1.5,-1.5) -- (-1,-1.5);
	\draw[magenta] (-1.5,-1.5) node[left] {$\SE_{p_0}$};
	\draw[dotted] (-1,-1)--(1,-1);
\end{tikzpicture}
\end{center}
\caption{}
\label{fig1}
\end{figure}

	\item[(H4)] All null geodesics have at most finite order contact with $\pd M$. 
\end{enumerate}
The recovery domain $\SD$ is then given by
\begin{equation}
	\label{eq:D.defn}
	\SD:=\{p\in M:\ol{\SE_p}\cap \pa M \subset (T_0,T)\times \pa M_0\}
\end{equation}
and we make the final assumption
\begin{enumerate}
	\item[(H5)] There exists $T_1\in [-T,T]$, such that $\{T_1\}\times M_0\subseteq \mathcal{D}$.
\end{enumerate}
\begin{figure}[h]
\begin{center}
	\begin{tikzpicture}[scale=1]
	\draw (-1,-2)--(1,-2);
	\draw (-1,2)--(1,2);
	\draw (0,-2) node[below] {$M_0$};
	\draw (1,-2)--(1,2);
	\draw (-1,-2)--(-1,2);
	\draw (1,-2) node[right] {$t=-T$};
	\draw (1,2) node[right] {$t=T$};
	\draw (1,0) node[right] {$\Sigma$};

	\draw (1,-1) node[right] {$t=T_0$};
	\draw[very thick,red] (1,-2)--(1,-1);
	\draw[very thick,red] (-1,-2)--(-1,-1);
	\draw (1,-1.5) node[right] {$\color{red}\Gamma$};

	\draw[dotted] (-1,-1)--(1,0);
	\draw[dotted] (-1,0)--(1,-1);
	\draw[dotted] (-1,1)--(1,2);
	\draw[dotted] (-1,2)--(1,1);
	\draw[dotted,pattern color=blue,pattern=north east lines] (0,-0.5)--(1,0)--(1,1)--(0,1.5)--(-1,1)--(-1,0)--cycle;

		\draw[->] (-1.5,0.5) -- (-1,0.5);
	\draw[blue] (-1.5,0.5) node[left] {$\mathcal{D}$};

	\draw[thick,blue] (-1,0.75)--(1,0.75);
	\draw (1,0.75) node[right] {$t=T_1$};

	\draw[dotted] (-1,-1)--(1,-1);

\end{tikzpicture}
\end{center}
\caption{}
\label{fig2}
\end{figure}
The hypotheses (H1)-(H2) were used in \cite{LO2} to construct a strictly pseudoconvex foliation of $\SE_p$, which is the key geometric ingredient in obtaining the unique continuation principle Proposition \ref{prop:UCP}.
The hypotheses (H3)-(H4) are used to obtain the exact controllability result of Proposition \ref{prop:control} on timeslices in the causal future of $p_0$, via the geometric control condition as in \cite{BLR}. 
The hypothesis (H5) imposes two pleasant structural properties on $\SD$, namely that $\SD$ is connected, and $\SD\cap \Sigma \neq \emptyset$.

The main result of this paper is as follows.
\begin{thm}\label{thm:main}
	Let $(M,g)$ be a Lorentzian manifold of the form \eqref{eq:productform}, \eqref{eq:metricform} satisfying hypotheses (H1-H5). 
	Let $\tilde g$ be a smooth metric lying in a sufficiently small $\SC^2$ neighbourhood of $g$.	Let $E$ be a rank $N$ Hermitian bundle over $M$ with the structure of a $G$-bundle for $G\subseteq U(\mathbb{C}^N)$ a fixed Lie subgroup of the Lie group of unitary transformations on $\CC^N$.

	Let $\nabla_1,\nabla_2$ be two connections on $E$ of the form \eqref{eq:B.def.0}, and let $V_1,V_2$ be two potentials satisfying \eqref{eq:potential.compat}.
	Let $\Lambda_1,\Lambda_2$ be the Dirichlet-to-Neumann maps defined in \eqref{eq:DtN} corresponding to the connection and potential pairs $(\nabla_1,V_1)$ and $(\nabla_2,V_2)$ respectively.

	Then if
	\begin{equation}\label{eq:equality.DtN}
		\Lambda_1=\Lambda_2 \textrm{ on $H_0^1(\Sigma;E)$},
	\end{equation}
	there exists $A\in \SC^\infty(\SD;G(E))$ such that $A^*\nabla_1 A=\nabla_2$ and $A^*V_1 A=V_2$ in $\SD$.
	Moreover, $A|_{\SD\cap \Sigma}=\Id$.
\end{thm}

\subsection{Examples}
\label{sec:examples}
The hypotheses (H1)-(H4) of Theorem \ref{thm:main} are equivalent to those in \cite{LO2}, which are strictly weaker than those in \cite{LO1}. As such, Theorem \ref{thm:main} applies to Hermitian vector bundles over the Lorentzian manifolds discussed in these works, with the additional hypothesis (H5) only constraining the size of the domain of recovery $\SD$ (see Figure \ref{fig2}). In particular, the following geometries are treated.
\begin{itemize}
	\item $M=[-T,T]\times \Omega$, where $\Omega \subset \RR^n$ is compact and connected, with nonempty interior and a smooth strictly convex boundary, equipped with the metric $\tilde{g}$ that is any sufficiently small $\SC^2$-perturbation of the Minkowski metric on $M$ \cite[Corollary~1.2]{LO2}.
	\item $M=[-T,T]\times M_0$, equipped with the metric $\tilde{g}$ that is any sufficiently small $\SC^2$-perturbation the ultrastatic metric $g(t,x)=-dt^2+g_0(x)$, where $(M_0,g_0)$ is a compact, simply connected Riemannian manifold with negative sectional curvature and smooth strictly convex boundary \cite[Section~3.1]{LO1}.
\end{itemize}

\subsection{Relation to existing literature}
We first give a brief account of the literature in the elliptic and scalar version of Theorem \ref{thm:main}, where $E$ is replaced by the trivial bundle $\pi:M\times \CC\to M$, the Lorentzian manifold $M$ is replaced by a compact, connected Riemannian manifold with boundary, and the connection wave operator $P=\Box+V$ is replaced by $\Delta+V$ where $\Delta$ is the Laplace-Beltrami operator and $V\in\SC^\infty(M)$ is a real potential.

There, the analogous question is to ask if for fixed metric $g$, the Dirichlet-to-Neumann map $\Lambda$ uniquely determines $V$. 
This is a geometric version of the classical Calder\'{o}n problem \cite{calderonoriginal}, where Alberto Calder\'{o}n raised the question of whether one can determine the electrical conductivity of a medium by making voltage and current measurements at the boundary. 

For $M$ a Euclidean domain, the case $\dim(M)\geq 3$ was solved affirmatively for Euclidean domains in \cite{sylvester.uhlmann} and the case $\dim{M}=2$ was treated in \cite{nachman}, see also \cite{bukhgeim}.
For arbitrary Riemannian manifolds and $V\in \SC^\infty(M)$, the problem remains wide open, however the case of real analytic $(M,g)$ and $V$ was solved in \cite{lee.uhlmann}.

Let us now turn to the scalar hyperbolic setting, where $(M,g)$ is Lorentzian of the form \eqref{eq:productform},\eqref{eq:metricform} and $P=\Box+V$. The Boundary Control (BC) method introduced by Belishev \cite{belishevBC} shows unique determination of $V$ from the Dirichlet-to-Neumann map $\Lambda$ under a variety of assumptions on $g$ and $V$. The ultrastatic case with $c=1$ and $g_0$ independent of $t$ was resolved in \cite{belishev.kurylev}, and the case of analytic $g$ and $V$ was resolved in \cite{eskin1},\cite{eskin2}.

These works all make use of the BC method, and consequently use the optimal unique continuation principal (UCP) of Tataru \cite{tataru}. As this optimal unique continuation principal breaks down for $\SC^\infty$ coefficients \cite{alinhac}, results are more scarce in this setting.

Some results for general $V\in \SC^\infty(M)$ include \cite{stefanov} which treat the case of general smooth $V$ for Minkowski spacetime, \cite{feizmohammadi} which treats ultrastatic spacetime under additional convexity assumptions on $M_0$, and \cite{spyros.lauri.miika} which treats the case of stationary spacetimes. These results relate the Lorentzian Calder\'{o}n problem to the injectivity of the light ray transform $\mathcal{L} V(\gamma)$ that maps inextendable null geodesics $\gamma$ to the integral of $f$ over $\gamma$. Such injectivity results are rare outside of the ultrastatic case.

In \cite{LO1},\cite{LO2}, unique determination of general $V\in \SC^\infty(M)$ was shown even for non-analytic $g$, provided that $(M,g)$ satisfies certain curvature bounds. Moreover, it was shown that the set of $g$ satisfying the curvature bounds had nonempty interior in $\SC^2$. The key novelty in \cite{LO2} was the observation that the geometric hypotheses (H1-H2) of Theorem \ref{thm:main} imply the existence of a strictly  pseudoconvex foliation in exterior nullcones $\SE_p$ by Lorentzian spheres centered at $p$. Such a foliation implies a unique continuation theorem for $P$ by classical theory (see for example \cite[Theorem~28]{Hormander7}), without the need for the stronger curvature assumptions in the earlier work \cite{LO1}. Using this UCP, together with ideas from the BC method, injectivity of $V\mapsto \Lambda_V$ was shown for the same broad class of metrics $\tilde g$ as in Section \ref{sec:examples}.

For the connection Laplacian on a vector bundle, the problem of recovering the coefficients of the connection (or the topology and geometry of the underlying bundle) up to gauge transformations from the Dirichlet-to-Neumann map has also been studied in many recent works which we shall now mention.

In the elliptic setting, the problem of recovering the coefficients of the connection and potential in $2$d up to gauge was completely solved in \cite{albin2013inverse}. Various partial results for the higher dimensional problem have been obtained. In \cite{eskin2001global}, the case of Euclidean domains is treated. The case of line bundles was studied in \cite{dos2009limiting},\cite{cekic2017calderon}, under the assumption that the base manifold was conformally transversally anisotropic. These works made additional geometric assumptions on $M$ and an assumption on the injectivity of the geodesic ray transform respectively.
Recovery of the connection and the geometry of the bundle was shown for Yang-Mills connections on Hermitian bundles of rank $m>1$ in \cite{cekic2020calderon}. Unlike the previous works, the analysis of the light ray transform does not play an essential role in this last result.
Under the assumption of analyticity of base manifold and coefficients, \cite{gabdurakhmanov2025calderon} recovers both the connection and the geometry of the bundle and base manifold in dimension $n > 2$. In dimension $2$, the bundle and connection are recovered for a fixed base manifold.

In the hyperbolic setting, the case of trivial vector bundles over Euclidean domains was resolved in \cite{eskin2005inverse}, and time-dependent Yang--Mills potentials were recovered in the followup work \cite{eskin2008inverse}.
The paper \cite{kurylev2018inverse} recovers the connection and geometry of the bundle and base manifold from only partial measurements, however all coefficients are assumed to be time-independent. 
Let us also mention the recent work \cite{spyros.lauri.miika} in which a matrix-valued potential is recovered from the source-to-solution map on stationary Lorentzian spacetimes under the assumption of a time-independent connection. This work proceeds by analysis of the light ray transform which seems difficult in the case of general time dependent coefficients.
Summing up, most existing results in both the elliptic and hyperbolic setting either require an assumption of analyticity, or strong geometric assumptions on the base manifold and vector bundle.

The present work generalises the investigations of \cite{LO1},\cite{LO2} to the setting of the connection wave operator \eqref{eq:box.def} acting on sections of a Hermitian vector bundle. The techniques used are similar to those in the scalar setting, however an essential difference in the higher rank case is the presence of a gauge invariance for the problem, as seen in the conclusion of Theorem \ref{thm:main}. 

\subsection{Organisation of the paper}
In Section \ref{sec:prelim}, we recall some prerequisite notions from Lorentzian geometry and study the form that $P$ takes in local trivialisations of $E$.
In Section \ref{sec:ucp}, we obtain a unique continuation principle Proposition \ref{prop:UCP} for solutions to \eqref{eq:wave.eq} based on classical Carleman estimate techniques. This relies on the existence of a strictly pseudoconvex foliation of the sets $\SE_p$ defined in \eqref{eq:exterior.nullcone.def}.
In Section \ref{sec:obscont}, we show well-posedness of the direct problem \eqref{eq:wave.eq}, and obtain an exact controllability result Proposition \ref{prop:control} for solutions to \eqref{eq:wave.eq}.
In Section \ref{sec:mainthm}, we complete the proof of Theorem \ref{thm:main}, making use of the the unique continuation principle and exact controllability result.
In Section \ref{sec:propagation}, we show that the classical propagation of singularities results for second-order differential operators on manifolds with non-characteristic boundary extend to the vector bundle setting.
In Section \ref{sec:observability}, we prove an observability estimate that played a key role in the proof of exact controllability in Section \ref{sec:obscont}. This result is based on the work \cite{BLR} in the setting of scalar wave operators in the presence of a geometric control condition.
In Section \ref{sec:goptics}, we construct Gaussian beam solutions for the connection wave operator that are concentrated along a null geodesic.
In Section \ref{sec:direct}, we include the proofs of energy estimates that were required for the study of the direct problem in Section \ref{sec:obscont}.

\subsection*{Acknowledgements}

The authors were supported by the European Research Council
of the European Union, grant 101086697 (LoCal),
and the Research Council of Finland, grants 347715,
353096 (Centre of Excellence of Inverse Modelling and Imaging)
and 359182 (Flagship of Advanced Mathematics for Sensing Imaging and Modelling).
Views and opinions expressed are those of the authors only and do not
necessarily reflect those of the European Union or the other funding
organizations.

\section{Preliminaries}
\label{sec:prelim}
\subsection{Lorentzian geometry}
\label{sec:lorentzian.geom}
We begin by recalling some elementary notions from Lorentzian geometry. 
Let $(M,g)$ be as in \eqref{eq:productform} and \eqref{eq:metricform}. For $p\in M$, we say $v\in T_pM$ is
\begin{itemize}
	\item \emph{spacelike} if $g(v,v)>0$;
	\item \emph{timelike} if $g(v,v)<0$;
	\item \emph{lightlike} if $g(v,v)=0$\textrm{ and }$v\neq 0$;
	\item \emph{causal} if $g(v,v)\leq 0$\textrm{ and }$v\neq 0$;
	\item \emph{future-pointing} if $g(v,\partial_t)>0$.
\end{itemize}
We shall use the same terminology to classify $v\in T_p^*M$ using the musical isomorphism to identify vectors with covectors.

We say a curve $\gamma\in \SC([a,b], M)$ is \emph{piecewise smooth future-pointing} if it is piecewise $\SC^\infty$, at every regular point we have that $\gamma'(t)$ is causal and future-pointing, and moreover that $g(\gamma'(t_+),\gamma'(t_-))<0$ at singular $t$, where $$\gamma'(t_{\pm}):=\lim_{h\to 0^{\pm}} \gamma'(t+h)\in T_{\gamma(t)}M.$$

If there exists a piecewise smooth future-pointing curve from $p$ to $q$ we write $p < q$ and say that $q$ is in the causal future of $p$. If such a curve exists that is timelike at regular points we write $p \ll q$ and say that $q$ is in the chronological future of $p$.
We denote by $\leq$ the minimal reflexive extension of $<$.

For $p\in M$, any covector in $T_p^*M$ is completely determined by its inner products with elements of $L_p^+$, the cone of future-pointing lightlike covectors.
\begin{lem}
	\label{lem:basic.lorentzian}
	For $\xi\in T_p^*M$, if $\ang{\xi,\eta}_{T_p^*M}=0$ for all $\eta\in L_p^+ M$, then $\xi=0$. 
\end{lem}
\begin{proof}
	This is an immediate consequence of $L_p^+M$ being a spanning set for $T_p^* M$ and the nondegeneracy of $g$.
\end{proof}
For $U\subseteq M$ open and $u=v\otimes\alpha\otimes \beta\in \SC^\infty(M;E^{m,n}\otimes T^*M\otimes T^*M)$, where
\begin{equation}\label{eq:tensor.power}
E^{m,n}:=E^{\otimes m}\otimes (E^*)^{\otimes n}	,
\end{equation}
we denote the metric contraction of $\alpha\otimes \beta$ by
\begin{equation}
	C(u)=\langle\alpha,\beta\rangle_g v \in \SC^\infty(M;E^{m,n}).
\end{equation}
Similarly to Lemma \ref{lem:basic.lorentzian}, we can determine endomorphism-valued $1$-forms using suitable metric contractions as in the following lemma. 
\begin{lem}
	\label{lem:basic.lorentzian.end}
	Let $(E_1,\ldots,E_N)$ be a local frame for $E$ near $p$. Then for $A\in \End(E_p)\otimes T_p^* M$, if
	$C(Au)=0$ for all $u=\sum_{j=1}^N \xi^j \otimes E_j$ with each $\xi^j\in L_p^+ M$, then $A=0$.
\end{lem}
\begin{proof}
	Writing $x_0=t$ for notational convenience, we let $A=\sum_{i=0}^n\sum_{j,k=1}^N A_{ijk}\, E_k \otimes E_j^*\otimes dx^i$. We then have
	\begin{equation}
		Au=\sum_{i,j=0}^n\sum_{k=1}^N A_{ijk}\, E_k\otimes dx^i\otimes \xi^j.
	\end{equation}
	Hence taking all but one $\xi^j$ equal to zero, the condition $C(Au)=0$ implies 
	\begin{equation}\label{eq:end.reduction}
		\ang{\sum_{i=0}^n A_{ijk}\, dx^i,\xi}_{T_p^* M}=0
	\end{equation} 
	for all $\xi \in L_p^+M$ and all $j,k$ and so we conclude $A=0$ by applying Lemma \ref{lem:basic.lorentzian}. 
\end{proof}

\subsection{Local computations}
\label{sec:local.comp}
We now study the form $P$ takes locally in terms of a local trivialisation of $E$ by an orthonormal frame, and record several identities that shall be used in what follows.

Let $(M,g)$, $E$ and $\nabla$ be as in Section \ref{sec:intro}. We restrict our attention to an open subset $U$ with local coordinates $(x_j)_{j=0}^n$ and corresponding coordinate frame $(e_j)_{j=0}^n$, where we once again have written $x_0=t$.

Let $(E_j)_{j=1}^N$ be a local $G$-frame for $E$ above $U$, giving rise to a $G$-frame $\phi:E|_U\to U\times \CC^N$.
As in \eqref{eq:B.def.0}, the connection is of the form
$$\nabla=\phi^*(d+B)(\phi^{-1})^* $$
for some $B=B_i\, dx^i$ where $B_i\in \mathcal{C}^\infty(U;\mathfrak{g})$.
In our local trivialisation the connection wave operator $\Box$ then takes the form
\begin{align*}
	(\phi^{-1})^*\Box\phi^*&=(d+B)^*(d+B)\\
	&=d^*d+d^*B+B^*d+B^*B.
\end{align*}
For $\omega\in\SC^\infty(U;\CC^N\otimes T^*U)$, we compute $B^*\omega=-C(B\omega)$, and so for $u\in\SC^\infty(U;\CC^N)$ we have
\begin{equation}
	\label{eq:top.two.terms}
	(\phi^{-1})^*\Box\phi^*u=d^*d-2C(B(du))+Zu
\end{equation}
for some smooth matrix-valued function $Z$ (that depends on $B$).

We extend the connection $\nabla$ to the dual bundle $E^*$ in the natural way, taking
\begin{equation}\label{eq:def.connection.dual}
(\nabla\mu,u)=d( \mu,u)-( \mu,\nabla u)
\end{equation}
as an equality of $1$-forms, for $\mu\in \SC^\infty(M;E^*)$ and $u\in \SC^\infty(M;E)$, where the pairing is the natural bilinear pairing between fibres of $E^*$ and $E$.
In terms of the local trivialisation $\tilde\phi:E^*|_U\to U\times \CC^N$ induced by the dual frame $(E_j^*)_{j=0}^N$ we have \begin{equation}\label{eq:coefficients.dual}
	\nabla=\left(\tilde\phi\right)^*(d+\overline B)\left({\tilde\phi}^{-1}\right)^*
\end{equation} where $\overline B$ denotes the entrywise conjugate of $B$.
We can then extend $\nabla$ to the tensor powers \eqref{eq:tensor.power} using the Leibniz rule to get a map
\begin{equation}\label{eq:tensor.connection}
\nabla: E^{m,n} \to E^{m,n} \otimes T^*M.
\end{equation}
An important special case is $E^{1,1}\cong \End(E)$. If $A\in \SC^\infty(\End(E))$ and $\tilde{A}\in\SC^\infty(U;\CC^N)$ is defined by 
\[\tilde A =(\phi^{-1})^*A \phi^*\]
then we have 
\begin{equation}\label{eq:covariant.A}
	\nabla A=\phi^*(d+[B,\tilde A])(\phi^{-1})^*.
\end{equation}

Equipping fibres and smooth sections of $E^{m,n}$ and $E^{m,n}\otimes T^*M$ with their natural bilinear forms, we can compute the adjoint $\nabla^*$, which is a connection version of $-\mathrm{div}$ and in local coordinates takes the form
\begin{equation}\label{eq:tensor.adjoint}
	\sum_{j=0}^n\nabla^*(T_j\, dx^j)=-\sum_{i,j=0}^nG^{-1}\nabla_i(g^{ij}GT_j)
\end{equation}
where $G=|\det g|^{1/2}$ is the Lorentzian volume density, $\nabla_i:=\nabla_{\pa_{x_i}}$, and $T_i\in\SC^\infty(M;E^{m,n})$.
From \eqref{eq:tensor.adjoint}, it follows that 
\begin{equation}
	\label{eq:tensor.laplacian}
	\nabla^*\nabla T = -\sum_{i,j=0}^n G^{-1}\pa_i (g^{ij}G) \nabla_j T -C(\nabla\nabla T)=-\sum_{i,j=0}^nG^{-1}\nabla_i(g^{ij}G \nabla_j T)
\end{equation}
and 
\begin{equation}\label{eq:products}
\nabla^*\nabla (T\otimes S)=\nabla^*\nabla T \otimes S + T \otimes \nabla^*\nabla S -2 C(\nabla T\otimes \nabla S)
\end{equation}
for $T,S\in\SC^\infty(M;E^{m,n})$. As contraction of a fixed factor $E\otimes E^*$ commutes with $\nabla$ and $\nabla^*\nabla$, it follows that for $A\in \SC^\infty(M;E^{1,1})\cong \SC^\infty(M;\End(E))$ and $u\in\SC^\infty(M;E)$, we have
\begin{equation}
\label{eq:products2}
\cl(Au)=(\cl A) u+A\cl u-2C(\nabla A \nabla u)
\end{equation}
and consequently
\begin{equation}\label{eq:products3}
	P(Au)=(P A) u+AP u-AVu-2C(\nabla A \nabla u).
\end{equation}

Finally, we note that the general Stokes theorem implies an analogue of the divergence theorem
\begin{equation}\label{eq:div.thm}
	(\nabla^*\omega,u)_E=(\omega,\nabla u)_{E\otimes T^*M}-(\mathrm{sgn}(\ang{\nu,\nu})\iota_\nu^* \omega,u)_{E,\pa M}
\end{equation}  
for $u\in \SC^\infty(M;E)$ and $\omega\in\SC^\infty(M;E\otimes T^*M)$
where $\nu$ is the outward pointing unit normal vector field on $\pa M$, and $(\cdot,\cdot)_{E,\pa M}$ is defined as in \eqref{eq:L2.def}, but with $M$ replaced by $\pa M$, and the volume density $dV_g$ replaced by its pullback under the inclusion $i:\pa M\injto M$.
From \eqref{eq:div.thm}, we obtain
\begin{equation}
	\label{eq:P.duality}
	(Pu,v)_E-(u,Pv)_E=-((\mathrm{sgn}\ang{\nu,\nu}\nabla_\nu u,v)_{E,\pa M}-(u,\mathrm{sgn}\ang{\nu,\nu} \nabla_\nu v)_{E,\pa M})
\end{equation}
for $u,v\in \SC^\infty(M;E)$.




\section{Unique continuation principle}
\label{sec:ucp}
In this section we use classical Carleman estimate techniques to deduce a unique continuation principle for $P=\Box+A$ with $A\in \mathrm{Diff}^1(M;E)$ arbitrary, where $\mathrm{Diff}^k(M;E)$ denotes the set of differential operators of order $k$ acting on sections of $E$ with coefficients in $\SC^\infty(M)$. The same strictly pseudoconvex foliation of $\SE_p$ constructed in \cite[Section~3]{LO2} can be used without modification.
\begin{prop}
\label{prop:UCP}
Let $(M,g)$ be of the form \eqref{eq:productform},\eqref{eq:metricform}, and let $E$ be a rank $N$ Hermitian vector bundle over $M$ equipped with a compatible connection $\nabla$. Suppose further that $g$ satisfies hypotheses (H1)-(H2) of Theorem \ref{thm:main} and let $\tilde{g}$ be a smooth Lorentzian metric on $M$ that lies in a sufficiently small $\SC^2(M)$-neighbourhood of $g$.
Let $$P=\Box+L\in \mathrm{Diff}^2(M;E)$$ where $L\in \mathrm{Diff}^1(M;E)$ is arbitrary and $\Box$ is the connection wave operator associated to $(\tilde g,\nabla)$.
Let $p\in \mathrm{int}(M)$ be such that $\SE_p\cap \partial M\subset \Sigma$, where $\SE_p$ is defined by \eqref{eq:exterior.nullcone.def} in the manifold $(M,\tilde{g})$ and $\Sigma=(-T,T)\times \partial M_0$.
Let $u\in H^{-s}(M;E)$ for some $s\geq 0$ be a distributional solution to
\begin{equation}
Pu=0\quad \textrm{on $\SE_p$}\label{eq:unique.continuation}.
\end{equation}
Suppose the traces $u|_\Sigma,\nabla_\nu u|_\Sigma$ both vanish on $\Sigma \cap \SE_p$. Then $u=0$ on $\SE_p$.
\end{prop}
\begin{rem}
	This is a straightforward generalisation of \cite[Theorem~1.3]{LO2} to the present setting of the connection wave operator acting on sections of a Hermitian vector bundle.
\end{rem}

\begin{proof}
We begin by embedding $M_0$ within a closed manifold $\tilde M_0$, and extending $\tilde{g},E,L$ and the connection $\nabla$ smoothly to $\tilde{M}:=[-T,T]\times \tilde M_0$. 

We now extend $u$ by zero to the distribution $u\in H^{-s}(\tilde{M};E)$ where $\tilde\SE_p$ is an open neighbourhood of $\SE_p$ in $\tilde M$. This extension remains a solution for the extended operator as $u|_{\Sigma\cap \SE_p}=\nabla_\nu u|_{\Sigma\cap \SE_p}=0$.

An application of propagation of singularities shows that $u\in \SC^\infty(\tilde \SE_p;E)$. Indeed, $u$ vanishes in $\tilde\SE_p\setminus\SE_p$, and for $q\in T^*\SE_p\setminus 0$, either $q\in\mathrm{ell}(P)$ or $q\in\mathrm{char}(P)$, where $\mathrm{char}(P)\subseteq T^*\tilde M$ is defined as in \cite[Definition~18.1.25]{Hormander3} and $\mathrm{ell}(P)=(T^*\tilde M\setminus 0)\setminus \mathrm{char}(P)$.

We have $\mathrm{ell}(P)\subseteq \WF(u)^c$ by microlocal ellipticity, and for $q\in\mathrm{char}(P)$, (H2) imples (see \cite[Lemma~5.1]{LO1}) that a segment of the null bicharacteristic through $q$ with either initial or terminal point $q$ meets $\tilde\SE_p\setminus M$ whilst remainining inside $\tilde\SE_p$. Hence, propapagation of singularities \cite[Theorem~2.1]{MR82i:35172} implies $q\in \WF(u)^c$, and we conclude that $u\in \SC^\infty(\tilde \SE_p;E)$.

We make use of the same strictly pseudoconvex foliation of $\SE_p$ as used in \cite[Proposition~3.5]{LO2}. The foliation is by level sets of $\psi=r_p$, the Lorentzian distance function associated to the point $p$ and the metric $\tilde g$.
It suffices to prove that for any $r>0$ and any $q\in \SE_p\cap \psi^{-1}(r)$, the vanishing of $u$ in $\{\psi > r\}\cap \tilde\SE_p$ implies the vanishing of $u$ in a neighbourhood of $q$.

In a small neighbourhood in $N\subseteq\tilde \SE_p$ of $q$, we now fix local coordinates $x=(x_0,\ldots,x_n)\in X\subset \RR^{n+1}$ for $M$ and a local trivialisation $\phi:E|_N\to X\times \CC^N$ of $E$ induced by a unitary frame. We may assume that $q$ is located at the origin in this coordinate chart, and by using these local coordinates and adjusting by an additive constant, we redefine $\psi$ to be a smooth function $X\to \RR$. 

For sections $v\in \SC^\infty(N;E|_N)$, we denote the $i$-th component of $(\phi^{-1})^* v$ by $v_i\in \SC^\infty(X;X\times \CC^N)$.

We write
\begin{equation}
	\psi_\epsilon= \sum_{|\alpha|\leq 2}x^\alpha \partial^\alpha \psi(0)/\alpha!-\epsilon|x|^2
\end{equation}
and choose $\epsilon,\delta >0$ sufficiently small so that in a small neighbourhood $X_\epsilon\subseteq X$ of $q$, we have strict pseudoconvexity of the level sets $\psi_\epsilon^{-1}(r)$ and
$\psi_\epsilon \leq \psi-\delta$ on $\partial X_\epsilon$.

As in the proof of \cite[Theorem~28.3.4]{Hormander7}, we have the following Carleman estimate in the set $Y=\{x\in X_\epsilon:\psi_\epsilon(x) > -\delta\}$ with weight $\phi=e^{\lambda \psi_\epsilon}$. 
\begin{equation}
\label{eq:carleman}
	\sum_{|\alpha|< 2} \tau^{2(2-|\alpha|)-1}\int_Y |D^\alpha w|^2 e^{2\tau \phi} \, dx\leq K(1+C/\tau^{1/2})\int_Y |\Box_{\tilde{g}} w|^2 e^{2\tau \phi}\, dx
\end{equation}
for $\tau>1$ and $w\in H_\mathrm{comp}^1(Y;Y\times \CC^N)$ with $\lambda $ sufficiently large.
Since $u\in \SC^\infty(\tilde \SE_p;E)$ is a solution to \eqref{eq:unique.continuation}, it follows from \eqref{eq:top.two.terms} that 
\begin{equation}
|(Pu)_i|\leq C\sum_{|\alpha|<2 }\sum_{j=1}^N |D^\alpha u_j|.
\end{equation}
for $1\leq i \leq N$. Take $\chi\in\SC_c^\infty(Y,\RR)$ with $0\leq \chi\leq 1$ and $\chi=1$ on the set $W=\{x\in X_\epsilon:\psi_\ep(x) \geq -\delta/2\}$ which has compact intersection with $\spt(u)$.
Then for $v=\chi u$ we have
\begin{equation}
	|(Pv)_i|\leq C\sum_{|\alpha|<2 }\sum_{j=1}^N |D^\alpha u_j|
\end{equation}
in $W$ for $1\leq i \leq N$ and so it follows from \eqref{eq:top.two.terms} that
\begin{equation}
	\|e^{\tau \phi}\Box_{\tilde g}v_i\|_{L^2(Y)}^2 \leq C'\sum_{|\alpha|<2 }\sum_{j=1}^N \|e^{\tau\phi }D^\alpha v_j\|^2_{L^2(W)}+C''e^{-\tau \delta }
\end{equation}
for $1\leq i \leq N$.
Applying \eqref{eq:carleman} to $w=(\phi^{-1})^*v$, we obtain
\begin{align}\label{eq:ucp.carleman}
\tau\sum_{|\alpha|<2 }\sum_{i=1}^N \|e^{\tau\phi}D^\alpha v_i\|^2_{L^2(W)} &\leq C \sum_{j=1}^N\|e^{\tau\phi}\Box_{\tilde g} v_j\|_{L^2(V)}^2 \\
&\leq C'\sum_{|\alpha|<2 }\sum_{j=1}^N \|e^{\tau\phi }D^\alpha v_j\|^2_{L^2(W)}+C''e^{-\tau \delta }.
\end{align}
Absorbing the first right-hand side term into the left-hand side and taking $\tau\to\infty$, we deduce that $v$ (and hence $u$) vanish on $W$ as required.
\end{proof}

\section{Direct problem and exact controllability}
\label{sec:obscont}
In this section we demonstrate well-posedness of the direct problem, and prove the exact controllability result Proposition \ref{prop:control} that plays a key role in the proof of Theorem \ref{thm:main}. 

First, we establish well-posedness of the direct problem in non-negative order Sobolev scales.
We introduce the data space
\begin{equation}
	\label{eq:data.space.def}
	\SX_s:= H_0^{s}(M;E)\times H_0^{s+1}(\Sigma)\times (H_0^{s+1}(M_0;E)\times H_0^s(M_0;E))
\end{equation}
and the corresponding solution space
\begin{equation}
	\label{eq:energy.space.def}
	\SY_s:=H^{s+1}(M;E)\cap C^1([-T,T];H^s(M_0;E))\cap C([-T,T];H^{s+1}(M_0;E))
\end{equation}
\begin{prop}
	\label{prop:forwardsmooth}
Let $(M,g)$ be of the form \eqref{eq:productform},\eqref{eq:metricform}, and let $E$ be a rank $N$ Hermitian vector bundle over $M$ equipped with a compatible connection $\nabla$. Let $(F,f,\phi)\in \SX_s$ for some $s\geq 0$.
Let $$P=\Box+L\in \mathrm{Diff}^2(M;E)$$ where $\Box$ is the connection wave operator associated to $(g,\nabla)$ and $L\in \mathrm{Diff}^1(M;E)$ is arbitrary.
Then there exists a unique solution 
\begin{equation}\label{eq:energy.space}
	u=\mathcal{S}(F,f,\phi)\in \SY_s
\end{equation}
to the problem
\begin{align}
	P u&=F\\
	u&= f \textrm{ on $\Sigma=(-T,T)\times \pd M_0$}\label{eq:wave.eq.2}\\
	(u,\nabla_t u)&=\phi \textrm{ on $\{T\}\times M_0$}
\end{align}
and the map $\mathcal{S}:\SX_s\to\SY_s$ is continuous. Moreover, $\nabla_\nu u|_{\Sigma}\in H^s(\Sigma;E)$.
\end{prop}
\begin{proof}
	The well-posedness of \eqref{eq:wave.eq.2} can be shown using standard techniques for treating hyperbolic boundary value problems, as in \cite[Chapter~24]{Hormander3}, with the energy estimate \eqref{eq:energy.main}, proven in Section \ref{sec:direct} below, taking the place of (24.1.4) in \cite{Hormander3}. The proof of local uniqueness and existence goes through essentially without modification in our bundle setting, working in local trivialisations induced by a choice of local orthonormal frame. The global existence and uniqueness with solution $u\in H^{s+1}(M;E)$ is then obtained as in \cite[Theorem~24.1.1]{Hormander3} in the case $\phi=0$. This also establishes the uniqueness claim in the proposition.
	
	To treat nonzero Cauchy data $\phi\in H_0^{s+1}(M_0;E)\times H_0^s(M_0;E)$, we first assume that $\phi\in \SC_c^\infty(\mathrm{int}(M_0))\times \SC_c^\infty(\mathrm{int}(M_0))$ and embed $M_0$ within a closed manifold $\tilde M_0$, extending $g_0,A,E$ and the connection $\nabla$ smoothly to $[-T,T]\times \tilde M_0$. We add a $\tilde \,$ to the notation for each of these extended objects.
	Existence of a solution for the Cauchy problem for $\tilde P=\Box_{\tilde \nabla,\tilde g}$ in the closed manifold $[-T,T]\times\tilde M_0$ with initial data $\phi$ then follows from \eqref{eq:energy.main} as in \cite[Theorem~23.2.4]{Hormander3}.	In particular, we can find a solution $v \in \SC^\infty([-T,T]\times \tilde M_0;E)$ to 
	\begin{align}
	P v&=0\\
	(v,\nabla_t v)&=\phi \textrm{ on $\{T\}\times \tilde M_0$}.
\end{align}
	By finite speed of propagation, we have $$\spt(v)\cap [T-\delta,T]\times \tilde M_0\subset [T-\delta,T]\times M_0$$ for sufficiently small $\delta >0$.
	Taking $\chi\in \SC_c^\infty(\RR)$ equal to $1$ in $[T-\delta/2,T+\delta/2]$, and supported in $(T-\delta,T+\delta)$, we can then identify $\chi v$ with an element $u_1\in \SC^\infty(M;E)$ that solves 
	\begin{align}
	P u_1&=F_1\\
	u_1&= 0 \textrm{ on $\Sigma=(-T,T)\times \pd M_0$}\\
	(u_1,\nabla_t u_1)&=\phi \textrm{ on $\{T\}\times \tilde M_0$}
\end{align}
	for $F_1=P(\chi v)=[P,\chi]v\in \SC_c^\infty(\mathrm{int}(M);E)$.
	The unique solution $u$ to \eqref{eq:wave.eq.2} is then obtained by taking the unique solution $u_2\in H^{s+1}(M;E)$ to the boundary value problem with zero Cauchy data
	\begin{align}
	P u_2&=F-F_1\\
	u_2&= f \textrm{ on $\Sigma=(-T,T)\times \pd M_0$}\\
	(u_2,\nabla_t u_2)&=0 \textrm{ on $\{T\}\times M_0$}
\end{align}
and setting $u=u_1+u_2$.

The preceding argument shows that for $F,f$ satisfying the hypotheses of the proposition and $\phi\in \SC_c^\infty(\mathrm{int}(M_0);E)$, there is a unique solution in $H^{s+1}(M;E)$ to \eqref{eq:wave.eq.2}. If $F$ and $f$ are smooth, then we additionally have the energy estimate \eqref{eq:energy.main}. 

Now let $(F,f,\phi)\in\SX_k$ for some $k\in \mathbb{N}$. Taking a sequence 
$$(F_j,f_j,\phi_j)\in  \SC^\infty(M;E)\times \SC_c^\infty(\Sigma;E)\times\SC_c^\infty(\mathrm{int}(M_0);E)^2$$
with 
$$(F_j,f_j,\phi_j)\to (F,f,\phi)\textrm{ in } \SX_k $$
we obtain from \eqref{eq:energy.main} that the corresponding solutions $u_j$ to \eqref{eq:wave.eq.2} (with $(F,f,\phi)$ replaced by $(F_j,f_j,\phi_j)$) are Cauchy in $\SY_k$. Hence their limit $u$ solves \eqref{eq:wave.eq.2}, is the unique such solution, and lies in $\SY_k$ with continuous dependence on data. We denote the solution operators by $\mathcal{S}_k:\SX_k\to\SY_k$. 

For any $k\in\NN$, the restriction of $\mathcal{S}_k$ to $\SX_s$ for $s\in [k,k+1]$ is then seen to be a continuous map $\mathcal{S}_s:\SX_s\to\SY_s$ by complex interpolation \cite{stein.interp}. This completes the proof.
\end{proof}
By time reversibility, \eqref{eq:wave.eq.2} is still well-posed if we replace the third condition with $$(u,\nabla_t u)=\phi\textrm{ on }\{T_0\}\times M_0$$
for any fixed $T_0\in [-T,T]$, provided that $$f\in H^{s+1}_0([-T,T_0]\times \pa M_0)+H^{s+1}_0([T_0,T]\times \pa M_0) .$$
We use the notation $$\mathcal{S}(F,f,\phi;T_0) $$
to denote the solution to this problem. 
\begin{rem}
	\label{rem:general.source}
	The above proof of well-posedness goes through without modification for  slightly more general source terms $F\in H_{\Sigma,T_0}^s(M;E)$, where $H_{\Sigma,T_0}^s(M;E)$ denotes the $H^s$-closure of $$\{u\in \SC^\infty(M;E):\mathrm{spt}(u)\cap (\{T_0\}\times M_0)=\emptyset\}.$$ 
\end{rem}

A standard duality argument using \eqref{eq:P.duality} gives the following well-posedness result in negative order Sobolev scales. The uniqueness part of the proof makes use the propagation of singularities result Corollary \ref{cor:propagation}, and so we impose an additional assumption on the principal symbol of $L-L^*$. We refer the reader to Section \ref{sec:propagation} for more details.
\begin{prop}
\label{prop:forward}
Let $(M,g)$, $E$, and $P$ be as in Proposition \ref{prop:forwardsmooth}, and suppose that $L\in\mathrm{Diff}^1(M;E)$ has the property that $L-L^*$ has scalar principal symbol. Let $s\geq 0$ and let $f\in H^{-s}(\Sigma;E)\cap \SE'((-T,T_0)\times \pd M_0;E)$ for some $T_0$ with $|T_0|<T$.
Then there exists a unique solution $u\in H^{-s}(M;E)$ to \eqref{eq:wave.eq}.
Moreover $\nabla_\nu u|_\Sigma\in H^{-s-1}(\Sigma;E)$ and $(u,\nabla_t u)|_{\{T_1\}\times M_0}\in H^{-s}(M_0;E)\times H^{-s-1}(M_0;E)$
for any $T_1\in [T_0,T]$.
\end{prop}
\begin{proof}
	First we show uniqueness. If $u_1,u_2\in H^{-s}(M;E)$ solve \eqref{eq:wave.eq} for a particular $f\in H^{-s}(\Sigma;E)\cap \SE'((-T,T_0)\times \pd M_0;E)$, then their difference $w=u_1-u_2$ is a solution with vanishing boundary data. Extending $w$ by $0$ to $[-T-\delta,-T)\times M_0$, for $\delta >0$ small, and extending $g,E,\nabla,L$ in an arbitrary smooth manner to this region, the initial conditions on $u_j$ imply that the extended $w$ remains a distributional solution on $[-T-\delta,T]\times M_0$ with vanishing boundary data. As every compressed generalised bicharacteristic of $P$ meets $[-T-\delta,-T)\times M_0$, where $w$ is identically zero, it follows from Corollary \ref{cor:propagation} that $\WF_b(w)=\emptyset$, and so $w\in\SC^\infty(M)$ from \cite[Theorem~18.3.27]{Hormander3}. It immediately follows that $w=0$ from the uniqueness part of Proposition \ref{prop:forwardsmooth}.

	Now, let $f\in \SC_c^\infty((-T,T_0)\times \pd M_0)$ and $F\in \SC_c^\infty$ be arbitrary. 
	Let $u=\mathcal{S}(0,f,0;-T)$ and let $v=\mathcal{S}(F,0,0;T)$.
	Substituting $u,v$ into \eqref{eq:P.duality}, we obtain 
	$$(u,F)_{E}=-(f,\nabla_\nu v|_{\Sigma})_E.$$
	Hence the map $F\mapsto -\nabla_\nu \mathcal{S}(F,0,0;T)|_\Sigma$ is the transpose of the map $f\mapsto \mathcal{S}(0,f,0;-T)$. As the former is a continuous map $H_0^{s+1}(M;E)\to H^s(\Sigma;E)$, it follows that the latter extends to a continuous map $$H^{-s}(\Sigma;E)\cap \SE'((-T,T_0)\times \pd M_0;E)\to H^{-s}(M;E).$$
	Similarly, the regularity of the trace $\nabla_\nu u|_\Sigma$ follows from that fact that the pairs of maps
	\begin{align}
		f&\mapsto \mathcal{S}(0,f,0;-T)\\
		g&\mapsto \mathcal{S}(0,g,0;T)\quad (g\in\SC_c^\infty(\Sigma;E))
	\end{align}
	are transpose to each other, and the regularity of the trace $(u,\nabla_t u)|_{t=T_1}$ follows from the fact that
	\begin{align}
		f&\mapsto (\nabla_t\mathcal{S}(0,f,0;-T),-\mathcal{S}(0,f,0;-T))|_{t=T_1}\\
		\phi&\mapsto \mathcal{S}(0,0,\phi;T_1)\quad (\phi\in \SC_c^\infty(\mathrm{int}(M_0);E)^2)
	\end{align}
	are transpose to each other.
\end{proof}
The exact controllability result we require is analogous to the result \cite[Proposition~6.1]{LO1} for scalar wave operators. As in the scalar setting, exact controllability can be obtained by combining a standard duality argument with an observability estimate. For us, this observability estimate is Proposition \ref{prop:observability}, proven in Section \ref{sec:observability} below.
\begin{prop}
\label{prop:control}
Let $(M,g)$ be of the form \eqref{eq:productform},\eqref{eq:metricform}, and let $E$ be a rank $N$ Hermitian vector bundle over $M$ equipped with a compatible connection $\nabla$. Suppose further that $g$ satisfies hypotheses (H2)-(H4) of Theorem \ref{thm:main} and let $\tilde{g}$ be a smooth Lorentzian metric on $M$ that lies in a sufficiently small $\SC^2(M)$-neighbourhood of $g$.

Let $$P=\Box+L\in \mathrm{Diff}^2(M;E)$$ where $\Box$ is the connection wave operator associated to $(\tilde g,\nabla)$ and $L\in \mathrm{Diff}^1(M;E)$ has the property that $L-L^*$ has scalar principal symbol.

Let $T_1\in [T_0,T]$.
Let $\phi\in H^{-s}(M_0;E)\times H^{-s-1}(M_0;E)$ with $s\geq 0$, compactly supported in the interior of $M_0$.

Then there exists $f\in H^{-s}(\Sigma;E)\cap \SE'((-T,T_0)\times \pd M_0;E)$ such that the unique solution $u\in H^{-s}(M;E)$ to \eqref{eq:wave.eq} satisfies 
\begin{equation}\label{eq:trace}
	(u,\nabla_t u)|_{t=T_1}=\phi.
\end{equation}
\end{prop}
\begin{proof}
	We take $\chi\in\SC_c^\infty((-T,T_0)\times \pa M_0)$ as in Proposition \ref{prop:observability}. For $\phi=(\phi_0,\phi_1)\in \SC_c^\infty(M_0;E)^2$ We then consider the map
	\begin{equation}
		\mathcal{T}:\phi\mapsto \chi\nabla_\nu\mathcal{S}(0,0,\phi;T_1)|_{\Sigma}.
	\end{equation}
For $g\in \SC^\infty(\ol{\Sigma};E)$, \eqref{eq:P.duality} implies that
\begin{align}
	\label{eq:control.duality}
(\mathcal{T}\phi,g)_E&=(\phi_1,\mathcal{S}(0,\chi g,0;-T)|_{t=T_1})_E-(\phi_0,\nabla_t \mathcal{S}(0,\chi g,0;-T)|_{t=T_1})_E.
\end{align}
From Proposition \ref{prop:forwardsmooth}, the operator $\mathcal{T}$ extends continuously to $$\mathcal{T}:H_0^{s+1}(M_0;E)\times H_0^s(M_0;E)\to H^s(\Sigma;E) .$$
From \eqref{eq:control.duality}, the transpose of $\mathcal{T}$ is given by
\begin{equation}
	g\mapsto (-\nabla_t\mathcal{S}(0,\chi g,0;-T),\mathcal{S}(0,\chi g,0;-T)).
\end{equation}
is the transpose of $\mathcal{T}$, and maps $$H^{-s}(\Sigma;E)\to H^{-s-1}(M_0;E)\times H^{-s}(M_0;E).$$
From Proposition \ref{prop:observability}, $\mathcal{T}$ is injective with closed range, and so the transpose is surjective as claimed.
\end{proof}

\section{Proof of Theorem \ref{thm:main}}
\label{sec:mainthm}
In this section, we complete the proof of Theorem \ref{thm:main}. Let $M,\tilde{g},E$ be as in the statement of this theorem, and fix connections $\nabla_{1},\nabla_2$ and potentials $V_1,V_2$ such that $\Lambda_1=\Lambda_2$. The main result of the section is Proposition \ref{prop:connections.equivalent}, establishing the gauge equivalence of $(\nabla_1,V_1)$ and $(\nabla_2,V_2)$ in the region
\begin{equation}\label{eq:region.recovery}
	\mathcal{D}=\{p\in M: \ol{\SE_p}\subseteq (T_0,T)\times M_0\}
\end{equation}
where $T_0$ is as in the hypothesis (H3) of Theorem \ref{thm:main}. 
\begin{prop}
\label{prop:gauge.equality}
	Suppose the hypotheses of Theorem \ref{thm:main} are satisfied and $p=(T_1,x_0)\in\mathrm{int}(\mathcal{D})$ for some $T_1\in (T_0,T)$.
	
	Suppose $\Lambda_1=\Lambda_2$, for some $V_1,V_2\in\SC^\infty(M)$ and connections  $\nabla_1,\nabla_2$. Then there exists $A(p)\in \End(E)|_p$ such that for any $f\in H_0^{\frac{n+1}{2}}(\Sigma)$ we have 
	\begin{equation}\label{eq:gauge.equality.at.p}
	u_f^{(1)}(p)=A(p)u_f^{(2)}(p),
	\end{equation}
	where $u_f^{(k)}$ are the solutions to \eqref{eq:wave.eq} guaranteed by Proposition \ref{prop:forward} and \ref{prop:forwardsmooth}.
\end{prop}
\begin{proof}
	Fix a local orthonormal frame $(E_1,\ldots,E_N)$ for $E$ near $p$ and define  $\delta_j \in H^{-\frac{n+1}{2}}(M_0;E)$ by $$(v,\delta_j)_{H_0^{\frac{n+1}{2}}\times H^{-\frac{n+1}{2}}(M_0;E)}=\ang{v(x_0),E_j(p)}_{E}$$ for $v\in \SC_c^\infty(M_0;E)$, where the left-hand side pairing is the sesquilinear duality pairing.

	An application of Proposition \ref{prop:control} shows the existence of $h_1,\ldots,h_N\in H^{-\frac{n-1}{2}}(\Sigma;E)\cap \SE'((-T,T_0)\times \pd M_0)$ such that $u_{h_j}^{(1)}$ satisfies 
	\begin{equation}
		(u_{h_j}^{(1)},\nabla_{t} u_{h_j}^{(1)})|_{t=T_1}=(0,\delta_j).
	\end{equation}
	Let $\mathscr{I}=\Span(h_j)$. From the support condition on $h_j$ in Proposition \ref{prop:control}, $u_{h_j}^{(1)}$ vanishes on $\Sigma\cap \SE_p$. From the finiteness of propagation speed it also follows that $\nabla_\nu u_{h_j}^{(1)}$ vanishes on $\Sigma \cap \SE_p$.
	
	Equality of the Dirichlet-to-Neumann operators $\Lambda_1,\Lambda_2$ then implies that $\nabla_\nu u_{h_j}^{(2)}$ vanishes on $\Sigma\cap \SE_p$, and by the unique continuation principle Proposition \ref{prop:UCP}, it follows that $u_{h_j}^{(2)}$ vanishes in $\SE_p$. In particular, the traces $(u_{h_j}^{(2)},\nabla_{t} u_{h_j}^{(2)})_{t=T_1}\in H^{-\frac{n-1}{2}}\times H^{-\frac{n+1}{2}}(M_0;E)$
	are supported at $x_0$, and so are both linear combination of the $\delta_j$ and their distributional derivatives.
	The Sobolev space membership implies 
	\begin{equation}
		(u_{h_i}^{(2)},\nabla_{t} u_{h_i}^{(2)})|_{t=T_1}=\sum_{j=1}^N a_{ij}^{(2)}(0,\delta_j).
	\end{equation}
	Using \eqref{eq:P.duality}, we can then compute
	\begin{align}
	& (\Lambda_k f,h_i)_{H^{\frac{n-1}{2}}\times H_0^{-\frac{n-1}{2}}(\Sigma;E)}-(f,\Lambda_k h_i)_{H_0^{\frac{n+1}{2}}\times H^{-\frac{n+1}{2}}(\Sigma;E)} \\
	&= (c^{-1/2}\nabla_{t} u_{f}^{(k)}|_{t=T_1},u_{h_i}^{(k)}|_{t=T_1})_{H^{\frac{n-1}{2}}\times H_0^{-\frac{n-1}{2}}(M_0;E)}\\
	&-(c^{-1/2}u_{f}^{(k)}|_{t=T_1},\nabla_{t} u_{h_i}^{(k)}|_{t=T_1})_{H^{\frac{n+1}{2}}\times H_0^{-\frac{n+1}{2}}(M_0;E)}\\
	&= -c^{-1/2}(p)\sum_{j=1}^N \ol{a_{ij}^{(k)}}(u_f^{(k)}|_{t=T_1},\delta_j)_{H^{\frac{n+1}{2}}\times H_0^{-\frac{n+1}{2}}(M_0;E)}\label{eq:pairing.lincomb}
	\end{align}
	where $a_{ij}^{(1)}=1$ for $i=j$ and $0$ otherwise, where the pairings are the sesquilinear duality pairings.
	Equality of the Dirichlet-to-Neumann maps $\Lambda_k$ implies \eqref{eq:pairing.lincomb} is independent of $k$, and it follows that for each $i$ we have
	$$(u_f^{(1)}|_{t=T_1},\delta_i)_{H^{\frac{n+1}{2}}\times H_0^{-\frac{n+1}{2}}(M_0;E)}=\sum_{j=1}^N \ol{a_{ij}}(u_f^{(2)}|_{t=T_1},\delta_j)_{H^{\frac{n+1}{2}}\times H_0^{-\frac{n+1}{2}}(M_0;E)}$$ 
	with coefficients $a_{ij}$ independent of $f$, and so $$A:=\sum_{i,j=1}^N \ol{a_{ij}} \, E_i\otimes E_j^*$$ satisfies the claims of the proposition.
\end{proof}
We shall also need the following lemma based on a Gaussian beam construction, the proof of which is deferred to Section \ref{sec:goptics}.
\begin{lem}
	\label{lem:goptics}
	For $p=(T_1,x_0)\in\mathcal{D}$ and $(E_1,\ldots,E_N)$ a local orthonormal frame near $p$, we can find ${f\in \SC_0^\infty((-T,T_1)\times \partial M_0)}$ such that
	\begin{enumerate}[i)]
		\item 	$u_f=\sum_{j=1}^N c_j E_j +o(1)$\label{eq:goptics1}
		\item $ du_f=i\lambda\sum_{j=1}^N c_j\omega^j \otimes E_j+O(1)$\label{eq:goptics2}
	\end{enumerate}
	at $p$ for any $c_j\in \CC$ and $\omega^j\in L_p^+M \subset T^* M$ as $\lambda\to\infty$.
\end{lem}

\begin{prop}
\label{prop:gauge.smooth}
$A\in \SC^\infty(\mathrm{int}(\mathcal{D});\End(E))$.
\end{prop}

\begin{proof}
	Let $p\in \mathrm{int}(\mathcal{D})$ and fix a local orthonormal frame $(E_1,\ldots,E_N)$ near $p$.	Using Lemma \ref{lem:goptics}, we can find $f_j\in \SC_0^\infty((-T,T_0)\times \pd M_0)$ for $j=1,\ldots,N$ such that $\{u_{f_j}^{(2)}(p)\}$ are linearly independent.
	
	Now we consider the $N\times N$ matrices $M^{(k)}$ with columns comprised of the coefficients of $u_{f_j}^{(k)}$ with respect to the frame $(E_1,\ldots,E_N)$.
	
	As the $u_{f_j}^{(k)}$ are smooth by Proposition \ref{prop:forward}, these matrices are smooth. Moreover it follows from linear independence of $\{u^{(2)}_{f_j}(p)\}$ that $M^{(2)}$ is invertible in a small neighbourhood $U$ of $p$. 
	
	As $M^{(1)}=M_AM^{(2)}$ where $M_A$ is the local matrix representation of $A$ with respect to the frame $(E_1,\ldots,E_N)$, we deduce that $M_A=M^{(1)}(M^{(2)})^{-1}$ and so $A$ is smooth in a neighbourhood of $p$.
\end{proof}
%

\begin{lem}
	\label{lem:D.connected}
	The region $\SD$ is connected.
\end{lem}
\begin{proof}
	First we note that \eqref{eq:D.defn} implies that $p\in \SD$ if and only if for every $z\in \pa M_0$, there exists $\epsilon > 0$ such that $(T_0+\epsilon,z) \leq p \leq (T-\epsilon,z)$. Indeed, if these causality conditions hold then every point $q\in [-T,T_0]\times \pa M_0$ satisfies $q \ll p $ by chronological transitivity, and so $[-T,T_0]\times \pa M_0\cap \ol{\SE_p}=\emptyset$. Similarly $\{T\}\times \pa M_0\cap \ol{\SE_p}=\emptyset$, so $p\in \SD$. Conversely, the violation of one of these causality statements for all $\epsilon > 0$ forces $\ol{\SE_p}\cap (\{T_0,T\}\times \pa M_0)\neq \emptyset$ and hence $p\notin \SD$.

	From this alternate characterisation of $\SD$, it immediately follows that if $p_-,p_+\in \SD$ with $p_-\leq p_+$, then we have $p\in \SD$ for all $p$ satisfying $p_-\leq p \leq p_+$. In particular, for $x\in M_0$ the set $\{t:(t,x)\in \SD\}$ is an interval containing $T_1$ by assumption (H5) of Theorem \ref{thm:main}. As $M_0$ is connected, we can conclude that $\SD$ is connected.
\end{proof}

\begin{prop}
\label{prop:unitary}
The section $A\in\SC^\infty(\mathrm{int}(\mathcal{D});\End(E))$ is a $G$-section.
\end{prop}
\begin{proof}
From Proposition \ref{prop:gauge.equality}, for arbitrary $f\in \SC_c^\infty(\Sigma)$ we have $P_1(Au_2)=P_2u_2=0$ in $\mathcal{D}$ where $u_j=u_f^{(j)}$ are the solutions to $\eqref{eq:wave.eq}.$

Taking $u=u_2$ and applying \eqref{eq:products3}, we then have
\begin{align}
	0&= P_1(Au)-AP_2u\\
	&= (P_1 A) u+A(P_1-P_2)u -AV_1u-2C(\nabla_1 A\nabla_1 u)\label{eq:det.calc}
\end{align}
pointwise, where $C$ denotes contraction of the factor $T^*M\otimes T^*M$ using the metric tensor $g$.

We now fix $p\in\mathrm{int}(\SD)$ and consider a local $G$-frame $(E_1,\ldots,E_N)$ for $E$ in a neighbourhood $U\subseteq M$ of $p$. Let $\phi:E|_U\to U\times \CC^N$ be the induced local trivialisation.

Then if we take
\[v=(\phi^{-1})^* (u|_U)\in \SC_c^\infty(U;\CC^N)\]
and
\[\tilde{A}=(\phi^{-1})^*(A|_U) \phi^*\in \SC_c^\infty(U;\CC^{N\times N}),\]
then \eqref{eq:det.calc} together with \eqref{eq:top.two.terms} and \eqref{eq:covariant.A} implies:
\begin{equation}
	\label{eq:before.killing.zeroth}
	C(d\tilde A+B_1\tilde A-\tilde AB_2)(dv))=Zv
\end{equation}
where $Z$ is a smooth matrix-valued function and the $B_j$ are as in \eqref{eq:B.def.0}, corresponding to the connections $\nabla_j$.

We then take $v=v_\lambda$ to be a family of Gaussian beam solutions for $P_2$ through $p$ using Lemma \ref{lem:goptics} and send $\lambda\to \infty$.
Noting that $Zv_\lambda=O_\lambda(1)$, \eqref{eq:before.killing.zeroth} implies
\begin{equation}
	\label{eq:after.killing.zeroth}
	C(d\tilde A+B_1\tilde A-\tilde AB_2)(dv))=0.
\end{equation}


As $dv(p)=i\lambda\sum_j \omega_j \otimes e_j+O_\lambda(1)$ with $\omega_j\in L_p^+$ arbitrary, an application of Lemma \ref{lem:basic.lorentzian.end} implies
\begin{equation}\label{eq:derivative.1}
d\tilde A=\tilde AB_2- B_1\tilde A
\end{equation}
at $p\in\mathrm{int}(\SD)$. As $p$ was arbitrary, \eqref{eq:derivative.1} holds identically on $\mathrm{int}(\SD)$. 

Now let $p=(T_1,x)\in \mathcal{D}\cap (-T,T)\times \pd M_0$, where $T_1$ is as in hypothesis (H5) of Theorem \ref{thm:main}.  Let $q\in \mathrm{int}(\SD)$ be arbitrary. Then by Lemma \ref{lem:D.connected}, we may choose a smooth curve $\gamma\in\SC^\infty([0,1],\SD)$ with $\gamma((0,1))\subset \mathrm{int}(\SD)$, $\gamma(0)=p$ and $\gamma(1)=q$. Along this curve, $\tilde A$ satisfies the differential equation
\begin{equation}
\label{eq:DE.curve}
 \tilde A'(t)=\tilde A(t)S_1(t)+S_2(t)\tilde A(t)
\end{equation}
where the $S_j$ take values in $\mathfrak{g}$, from \eqref{eq:B.def.0} and the assumption that $B_j\in \mathfrak{g}\subseteq{u}(N)$.
Smoothness of $S_1,S_2$ in $M$ allows us to extend $\tilde A$ continuously to the endpoints, so by taking the limit of \eqref{eq:gauge.equality.at.p} along $\gamma$ for some fixed $f\in \SC_c^\infty(\Sigma)$, we obtain $\tilde A(0)f(p)=f(p)$. As this can be done for any choice of $f$, it follows $\tilde A(0)=\Id$.

This initial condition together with \eqref{eq:DE.curve} implies that $\tilde A(t)$ takes values in the corresponding Lie group $G$. Indeed, following a standard technique for dealing with non-autonomous systems (see for example \cite[p.~69]{kobayashi}), the unique solution to \eqref{eq:DE.curve} must coincide with the unique solution $\tilde A\in \SC^\infty([0,1],G)$ to $\tilde A'=\tilde AS_1+S_2\tilde A$ with $\tilde A(0)=\Id$, that is the first component of the integral curve from $\Id$ of the vector field $V(A,t)=(AS_1(t)+S_2(t)A,1)$ on $G\times \RR$.
\end{proof}

\begin{prop}
\label{prop:connections.equivalent}
In the region $\mathcal{D}$ defined in \eqref{eq:region.recovery}, we have
\begin{equation}
A^*\nabla_1 A=\nabla_2 \textrm{ and }A^*V_1 A=V_2.
\end{equation}
\end{prop}
\begin{proof}
Define $P_3=A^* P_1 A$, where $A\in\SC^\infty(\SD;U(E))$ is as in Proposition \ref{prop:unitary}. It is immediate that $P_3$ is of the form \eqref{eq:box.def} with connection $\nabla_3=A^*\nabla_1 A$ and potential $V_3=A^* V_1 A$. We shall show $\nabla_2=\nabla_3$ and $V_2=V_3$.

Applying Proposition \ref{prop:gauge.equality}, we get that $P_3(u_f^{(2)})=P_2(u_f^{(2)})=0$ for any $f\in H_0^{\frac{n+1}{2}}((-T,T_1)\times \partial M_0;E)$.
We now take fix $p\in\mathrm{int}(\SD)$ and a local frame $(E_1,\ldots,E_N)$ for $E$ near $p$, and again take $u=u_f^{(2)}$ as in Lemma \ref{lem:goptics} parametrised by $\lambda\to\infty$.

As in \eqref{eq:before.killing.zeroth}, we compute 
\begin{equation}\label{eq:recover.connection}
	0=(P_3-P_2)(u)=2C(( B_2- B_3)(du))
\end{equation}
modulo $O_{\lambda}(1)$ at $p$ as $\lambda\to \infty$, where the $B_j$ are as in \eqref{eq:B.def.0} corresponding to the connections $\nabla_j$.
As $du(p)=i\lambda\sum_j \omega^j \otimes E_j+O_\lambda(1)$ with $\omega^j\in L_p^+$ arbitrary, division by $\lambda$ and an application of Lemma \ref{lem:basic.lorentzian.end} implies that $B_2(p)=B_3(p)$, and as $p$ is arbitrary, we get
\begin{equation}
	\nabla_2=\nabla_3\textrm{ in } \SD.
\end{equation}
Furthermore, the equation $(P_3-P_2)u=0$ now reduces to 
\begin{equation}
	\label{eq:recover.potential}
(V_3-V_2)u=0.
\end{equation}
at any $p\in \mathrm{int}(\SD)$.
The identity \eqref{eq:recover.potential} holds at $p$ for any choice of $u(p)$ by another application of Lemma \ref{lem:goptics},  so we can conclude that $V_2=V_3$ in $\SD$.
\end{proof}

\section{Propagation of singularities}
\label{sec:propagation}

In this section, we collect the required results on propagation of singularities for second-order differential operators on vector bundles over manifolds with non-characteristic boundary. These results are classical and well-known in the scalar case, but a comprehensive treatment of the extension of these results to the vector bundle setting does not appear in the literature, to the authors knowledge.

The classical theorem on propagation of singularities for operators of real principal type on closed manifolds goes back to Duistermaat--H\"{o}rmander \cite{Duistermaat-Hormander1}.

\begin{thm}\label{thm:propagation.scalar.1}
Let $M$ be a closed manifold and let $P\in\mathrm{Diff}^m(M)$ have real principal symbol $p\in S^m(T^*M)$. Let $u\in \SD'(M)$ be a solution to $Pu=f\in\SD'(M)$. Then \[\WF(u)\setminus \WF(f)\subseteq \mathrm{char}(P)\subseteq T^*M\setminus 0\] is a union of inextendable integral curves for the Hamiltonian vector field \[H_p=\partial_\xi p \cdot\partial_x-\partial_x p\cdot \partial_\xi \] in $\mathrm{char}(P)\setminus \WF(f)$. 
\end{thm}
Now suppose $M$ is a manifold with smooth non-characteristic boundary, $P\in \mathrm{Diff}^2(M)$, and $u\in \dot\SD'(M)$ is a solution of the boundary value problem
\begin{align}
	P u&=f \textrm{ in int($M$)}\\
	u&= u_0 \textrm{ on $\pd M$},\label{eq:general.bvp}
\end{align}
where $f\in \dot\SD'(M)$ and $u_0\in \SD'(\pa M)$. Here $\dot\SD'$ denotes the class of \emph{supported distributions}, as defined in \cite[Appendix~B]{Hormander3}.

Analogous results to Theorem \ref{thm:propagation.scalar.1} are then phrased using the compressed cotangent bundle $T_b^* M$, and the $b$-wavefront set $\WF_b(u)\subset T_b^* M$ of Melrose defined in \cite[Definition~18.3.25]{Hormander3} using the calculus $\Psi_b(M)$ of $b$-pseudodifferential, or totally characteristic operators.

We make the standing assumption that the bicharacteristics of $P$ meet $\partial (T^*M)$ to at most finite order, so that there is a uniquely defined compressed generalised bicharacteristic flow of $P$ \cite[Chapter.~18.3]{Hormander3} on $T_b^*M$. 

The natural map $\rho:T^*M \to T_b^* M $ is a bundle isomorphism over $\mathrm{int}(M)$ such that the pullback of the bicharacteristic flow in $T^*M$ is the compressed generalised bicharacteristic flow in $T_b^*M$ over $\mathrm{int}(M)$.

Points $\gamma\in T^*(\pa M)\setminus 0\subset T_b^* M|_{\pa M}$ are classified by the cardinality of $\rho^{-1}(\gamma)\cap \mathrm{char}(P)$. 
The elliptic set, hyperbolic set, and glancing set are respectively given by
\begin{align}
		\mathscr{E}&:=\{\gamma\in T^*(\pa M)\setminus 0:|\rho^{-1}(\gamma)\cap \mathrm{char}(P)|=0\}\\
		\mathscr{H}&:=\{\gamma\in T^*(\pa M)\setminus 0:|\rho^{-1}(\gamma)\cap \mathrm{char}(P)|=2\}	\label{eq:boundary.classification}\\
		\mathscr{G}&:=\{\gamma\in T^*(\pa M)\setminus 0:|\rho^{-1}(\gamma)\cap \mathrm{char}(P)|=1\}
\end{align}


We additionally assume $f\in\SN(M)$, defined in \cite[Definition~18.3.30]{Hormander3} as 
\begin{equation}
	\SN(M):=\{u\in \SA'(M):\WF_b(u)|_{\pa M}\subseteq T^*(\pa M)\}.
\end{equation}
where $\SA'(M)\subseteq \dot\SD'(M)$ is the class of supported distributions that act continuously on test functions $\phi\in\SC_c^\infty(M)$ with respect to the topology of the space of conormal distributions $I^m(M,\pa M)$
for all $m\geq -(n+2)/4$. 

\begin{thm}\label{thm:propagation.scalar.2}
	Let $M$ be a closed manifold and let $P\in\mathrm{Diff}^2(M)$ have real principal symbol $p\in S^2(T^*M)$, such that $\pa M$ is non-characteristic for $P$ and the bicharacteristics of $P$ meet $\pa (T^*M) $ to at most finite order.
	Let $u_0\in \SD'(\pa M)$, $f\in \mathcal{N}(M)$, and suppose that $u\in \dot \SD'(M)$ is a solution to \eqref{eq:general.bvp}.
	Then every $\gamma\in \WF_b(u)\setminus(\WF_b(f)\cup \WF(u_0))$ lies on a compressed generalised bicharacteristic of $P$, either above $\mathrm{int}(M)$ or in $\SCH\cup \SCG$. In either case, an open neighbourhood of $\gamma$ in the bicharacteristic also lies in $\WF_b(u)$.
\end{thm}
For $\gamma$ above $\mathrm{int}(M)$, Theorem \ref{thm:propagation.scalar.2} reduces to Theorem \ref{thm:propagation.scalar.1}. For $\gamma$ above $\pa M$, the situation is more delicate, and is divided into cases according to whether the compressed generalised bicharacteristic curve through $\gamma$ corresponds to a reflected, grazing, or gliding bicharacteristic curve in $T^*M$ upon pulling back by the natural map $ T^* M \to T_b^*M$.

The case of reflected singularities was treated in \cite{lax.nirenberg}, see also \cite{taylor.reflection}.
The case of grazing singularities was treated in \cite{Taylor1} in the case of second order tangency and in \cite{melrose.sjostrand} more generally. The case of gliding singularities is treated in \cite{melrose.gliding}. A detailed treatment of all cases can be found in \cite[Chapter~24]{Hormander3}.

Theorem \ref{thm:propagation.scalar.1} and Theorem \ref{thm:propagation.scalar.2} can be extended to a broad class of second-order differential operators on vector bundles, including those of this paper. Theorem \ref{thm:propagation.scalar.1} is explicitly stated for systems in \cite[Theorem.~2.1]{MR82i:35172}, whilst a bundle-valued version of Theorem \ref{thm:propagation.scalar.2} follows from the proof of the scalar result \cite[Chapter.~24]{Hormander3} with only minor modifications. It is convenient to formulate this result for operators acting on half-densities. 
\begin{prop}
\label{prop:propagation1}
Let $P\in\mathrm{Diff}^2(M;E\otimes \Omega^{1/2})$, where $E$ is a rank $m$ Hermitian vector bundle over $M$, a smooth manifold with boundary. 
Suppose further that
\begin{enumerate}\label{eq:propagation.hypotheses}
	\item $P$ has real and scalar principal symbol 
	\item $P-P^*$ has scalar principal symbol.
	\item $\pa M$ is noncharacteristic with respect to $P$.
\end{enumerate}	
Let $u\in\mathcal{N}(M;E\otimes \Omega^{1/2})$ be a solution of the boundary value problem \eqref{eq:general.bvp}, where $f\in \mathcal{N}(M;E\otimes \Omega^{1/2})$ and $u_0\in \SD'(\pd M;E\otimes \Omega^{1/2})$.
Then every $\gamma\in \WF_b(u)\setminus(\WF_b(f)\cup \WF(u_0))$ lies on a compressed generalised bicharacteristic of $P$ , either above $\mathrm{int}(M)$, or in $\SCH\cup \SCG$. In either case, an open neighbourhood of $\gamma$ in the bicharacteristic also lies in $\WF_b(u)$.
\end{prop}
A special case of Proposition $\ref{prop:propagation1}$ is the following result, which is the one we make direct use of in this paper.

\begin{cor}
	\label{cor:propagation}
	Let $(M,g)$ be of the form \eqref{eq:productform},\eqref{eq:metricform}, and let $E$ be a rank $N$ Hermitian vector bundle over $M$ equipped with a compatible connection $\nabla$.
	Let $$P=\Box+L\in\mathrm{Diff}^2(M;E)$$
	where $L\in\mathrm{Diff}^1(M;E)$ has the property that $\sigma_1(L-L^*)\in S^1(T^*M;\End(E))$ is scalar. 
	Let $u\in\mathcal{N}(M;E)$ be a solution of the boundary value problem \eqref{eq:general.bvp} for $P$, where $f\in \mathcal{N}(M;E)$ and $u_0\in \SD'(\pd M;E)$.
Then every $\gamma\in \WF_b(u)\setminus(\WF_b(f)\cup \WF(u_0))$ lies on a compressed generalised bicharacteristic of $P$ , either above $\mathrm{int}(M)$, or in $\SCH\cup \SCG$. In either case, an open neighbourhood of $\gamma$ in the bicharacteristic also lies in $\WF_b(u)$.
\end{cor}
\begin{proof}
	To the operator $P$, we can associate the operator $\tilde P$ acting on half densities by
	\begin{equation}
		\tilde P (u |dx|^{1/2}):=G^{1/2}P(G^{-1/2}u ) \, |dx|^{1/2}
	\end{equation}
	in local coordinates, where $G=|\det g|^{1/2}$. The operator $\tilde{P}$ is then formally self-adjoint and so clearly satisfies the hypotheses of Proposition \ref{prop:propagation1}.
\end{proof}
In the remainder of this section we discuss the proof of Proposition \ref{prop:propagation1}. For $\gamma$ above $\mathrm{int}(M)$, we can directly apply \cite[Theorem.~2.1]{MR82i:35172}, so we need only consider $\gamma$ above $\pa M$. For such $\gamma$, we check that the assumptions on $P$ allow us to reduce to the proof in the scalar case that is contained in \cite[Chapter~24]{Hormander3}. In this text, the proof is divided into the cases according to whether $\gamma$ is elliptic, hyperbolic, or glancing. The glancing region $\SCG$ is the most subtle, and is subdivided into diffractive and non-diffractive covectors. 
As the results are microlocal in nature, we can without loss of generality fix a choice of local coordinates for $M$ and local trivialisation for $E$ from the outset. Elements of $\Psi(M;E)$ (resp. $\Psi_b(M;E)$) can then be locally identified with $m\times m$ matrices of operators in $\Psi(\RR^n)$ (resp. $\Psi_b(\ol{\RR_+^n})$), and the matrix-valued principal symbols in $S^k(\RR^{2n};\CC^{m\times m})$ (resp. $S^k(\RR_+^n\times \RR^n;\CC^{m\times m})$) can be taken entrywise.

 We work in the local coordinates introduced in \cite[Section~C.5]{Hormander3} generalising boundary normal coordinates. In these coordinates $(x_1,\ldots,x_n)\in U \subset \RR^n$, where $x_1$ is a boundary-defining function for $\partial M$ and the interior of $M$ is given by $U\cap \{x_1 > 0\}$. The local trivialisation of $E$ over $U$ is chosen to be given by a orthonormal frame, shrinking $U$ if necessary.

From \cite[Lemma~C.5.3]{Hormander3}, up to sign, the operator $P$ takes the form
\begin{equation}
	P= D_1^2- R_2(x,D')-\sum_{j=1}^n A_j(x)D_j - R_0(x)
\end{equation}
where the coefficients $A_j,R_0$ lie in $\mathcal{C}^\infty(U;\CC^{m\times m})$ and $R_2$ is a smooth matrix of tangential differential operators with $R_2(x,\xi')$ homogeneous of degree $2$ in $\xi'$. From the fact that the trivialisation is orthonormal, it follows that the hypotheses on $P$ in Proposition \ref{prop:propagation1} are equivalent to the assertion that $\sigma(R_2)$ is real and scalar, and each $A_j-A_j^*$ is scalar.

In fact, we can assume without loss of generality that $A_1=0$. To see this, we conjugate $P$ by a smooth endomorphism $B\in \SC^\infty(U;\CC^{m\times m})$. The corresponding coefficient of $D_1$ in the operator $B^{-1}P B$ is 
\[2B^{-1}D_1 B-B^{-1}A_1B\]
which vanishes if we take $B$ to be the solution of the ODE $$D_1 B=\frac12 A_1B,\quad  B|_{\{x_1=0\}}=\mathrm{Id}. $$ This solution $B$ is a section of the general unitary bundle $$GU(m)=U(m)\otimes \RR^+,$$ as $iA_1\in\mathfrak{gu}(m)=\mathfrak{u}(m)\otimes \RR$. It follows that the conjugated operator $B^{-1}PB$ still satisfies the hypotheses of Proposition \ref{prop:propagation1}, and $v=B^{-1}u$ locally solves \eqref{eq:general.bvp} with $v$ replacing $u$ and $B^{-1}f$ replacing $f$. The conclusion of the proposition for $v$ will then imply the conclusion for $u$, as $B$ is smooth and invertible.

Summing up, we can assume that \begin{equation}
	P=D_1^2-R(x,D')
\end{equation}
where $R$ is a matrix of tangential differential operators of order $2$ such that $R$ has a real and scalar principal symbol $r$, and $R-R^*$ has a scalar principal symbol.

We note that if $A,B$ are principally scalar pseudodifferential operators acting on sections of a vector bundle with orders $k,l$ respectively, then $[A,B]$ is also principally scalar with principal symbol $[\sigma(A),\sigma(B)]\in S^{k+l-1}$. As only the principal and subprincipal symbols of $P$ play any role in the proof of \cite[Theorem~24.5.3]{Hormander3}, it will follow that only principally scalar operators show up when generalising to the present bundle setting.
The proof is divided into three cases based on the partitioning \eqref{eq:boundary.classification}.

The elliptic region $\mathscr{E}$ is given in local coordinates by \[\mathscr{E}=\{(x,\xi):x_1=0, r(x,\xi') < 0\}.\] In this region, \cite[Theorem~20.1.14]{Hormander3} applies directly, and gives

\begin{prop}
	\label{prop:elliptic.case}
If $u$ solves the boundary value problem \eqref{eq:general.bvp}, then
\begin{equation}
\label{eq:elliptic.case}
	\mathscr{E}\cap \WF_b(u)\subseteq \mathscr{E}\cap(\WF_b(f)\cup \WF(u_0)).
\end{equation}
\end{prop}

The hyperbolic region $\mathscr{H}$ is given in local coordinates by \[\mathscr{H}=\{(x,\xi):x_1=0, r(x,\xi') > 0\}.\] In this region we make use of the following lemma.

\begin{lem}\label{lem:factorisation.hyperbolic}
	Let $(0,\xi_0')\in\mathscr{H}$. Then there exist operators $\Lambda=\Lambda(x,D')\in \Psi^1(\RR^{n-1};\CC^m)$ and $B_{\infty}\in\Psi^{-\infty}(\RR^{n-1};\CC^m)$ smoothly dependent on the parameter $x_1$ such that
	\begin{equation}
		P=(D_1-\Lambda)(D_1+\Lambda)+B_{\infty}
	\end{equation}
	microlocally near $(0,\xi_0')$. 
\end{lem}
\begin{proof}
	Let $\chi(x,\xi')$ be a smooth cutoff to a small neighbourhood of $(0,\xi_0')$ so that $r \geq  \epsilon\ang{\xi'} $ on $\spt(\chi)$ for some $\epsilon > 0$.
	By a direct computation, we have 
\begin{equation}\label{eq:principal.factorisation}
		\Op(\chi)P=\Op(\chi)\left(D_1-\Op\sqrt{r(x,\xi')}\right)\left(D_1+\Op\sqrt{r(x,\xi')}\right)+B_{-1}
\end{equation}
	for some $B_{-1}\in \Psi^1(\RR^{n-1};\CC^m)$ smoothly dependent on the parameter $x_1$. 
	We now inductively choose \[c_j\in S^{-j}(\RR^{n}\times \RR^{n-1};\CC^m)\] so that the operator
	\begin{equation}
		\Lambda_k=\Op\left(\sqrt{r(x,\xi')}+\sum_{j=0}^k c_j\right)
	\end{equation}
	satisfies
	\begin{equation}
		\label{eq:induction.factorisation}
		\Op(\chi)(P-(D_1-\Lambda_k)(D_1+\Lambda_k))=B_k\in \Psi^{-k}(\RR^{n-1};\CC^m)
	\end{equation}
	for every $k \geq -1$, with $B_k$ smoothly dependent on the parameter $x_1$.
	Indeed, the case $k=-1$ is trivial from \eqref{eq:principal.factorisation}, and if we have selected $c_j\in S^{-j}(\RR^{n}\times\RR^{n-1};\CC^m)$ for $0\leq j \leq k$ such that \eqref{eq:induction.factorisation} holds, then for arbitrary $c_{k+1}\in S^{-k-1}(\RR^n\times\RR^{n-1};\CC^m)$, we have
	\begin{align}
		& P-(D_1-\Lambda_{k+1})(D_1+\Lambda_{k+1})\\
		&= P-(D_1-\Lambda_{k}-\Op(c_{k+1}))(D_1+\Lambda_{k}+\Op(c_{k+1}))\\
		&= B_k-[D_1,\Op(c_{k+1})]+\Lambda_k \Op(c_{k+1})+\Op(c_{k+1})\Lambda_k
	\end{align}
	which has principal symbol $b_k+2\sqrt{r(x,\xi')}c_{k+1}\in S^{-k}(\RR^n\times\RR^{n-1};\CC^m)$. Taking 
	\[c_{k+1}=-\frac{b_k}{2\sqrt{r(x,\xi')}}\]
	then establishes \eqref{eq:induction.factorisation} with $k$ replaced by $k+1$. We can then take $\Lambda=\Op(\sqrt{r}+c)$, where $c\in S^{0}(\RR^n\times\RR^{n-1};\CC^m)$ is a formal resummation of the $c_j$.
\end{proof}

\begin{prop}
\label{prop:hyperbolic.case}
	If $u$ solves the boundary value problem \eqref{eq:general.bvp}, and $\gamma=(x',\xi')\in (\mathscr{H}\cap \WF_b(u))\setminus (\WF_b(f)\cup \WF(u_0))$, then an open neighbourhood of $\gamma$ in the compressed generalised bicharacteristic through $\gamma$ lies in $\WF_b(u)$.
\end{prop}
\begin{proof}
	The proof of \cite[Theorem~24.2.1]{Hormander3} carries over directly, using Lemma \ref{lem:factorisation.hyperbolic} to factorise $P$ as in \cite[(24.2.5)]{Hormander7}.
\end{proof}

The diffractive region is given by
\begin{equation}
	\SCG_d:=\{(x,\xi):x_1=0,r(x,\xi')=0,\pa_{x_1}r(x,\xi) > 0\}.
\end{equation}

\begin{prop}
\label{prop:diffractive.case}
	If $u$ solves the boundary value problem \eqref{eq:general.bvp} and $\gamma=(x',\xi')\in (\mathscr{G}_d\cap \WF_b(u))\setminus (\WF_b(f)\cup \WF(u_0))$, then an open neighbourhood of $\gamma$ in the compressed generalised bicharacteristic through $\gamma$ lies in $\WF_b(u)$.
\end{prop}
\begin{proof}
	The proof in the scalar case \cite[Theorem~24.4.1]{Hormander3} is based on a positive commutator estimate, using the identity
	\begin{equation}\label{eq:comm.identity.1}
		2\mathrm{Im}\ang{Pu,Qu}_X=\sum_{j,k=0}^1 \ang{B_{jk}(x',D')D_1^ku,D_1^ju}_{\pa X}+\sum_{j,k=0}^1\ang{C_{jk}(x,D')D_1^ku,D_1^ju}_X
	\end{equation}
	where $X=\ol{\RR_+^n}$, $u\in\SC_c^\infty(X)$, and where 
	\begin{equation}
		Q(x,D)=Q_1(x,D')D_1+Q_0(x,D').
	\end{equation}
	Here $Q_j$ is a pseudodifferential operator in $x'$ of order $-j$ and the formal adjoints satisfy $Q_1=Q_1^*$ and $Q_0-Q_0^*=[D_1,Q_1]$, so in particular $Q=Q^*$.
	We have $B_{11}=Q_1,B_{01}^*=B_{10}=Q_0,B_{00}=Q_1(R+R^*)/2$ for $x_1=0$, the principal symbol $c_{jk}$ of $C_{jk}$ is real of order $1-j-k$, $c_{01}=c_{10}$ and 
	\begin{equation}
		\sum_{j,k=0}^1 c_{jk}(x,\xi')\xi_1^{j+k}=\{p,q\}-q\sigma(R-R^*)/i
	\end{equation}
	where $p$ and $q$ are the principal symbols of $P$ and $Q$ respectively.

	In the proof of \cite[Theorem~24.4.1]{Hormander3}, the operator $Q$ is carefully constructed so that its principal symbol $q$ is decreasing along the generalised bicharacteristic flow, with an additional strict negativity condition on $H_pq$ that is exploited using the sharp G\r{a}rding inequality for systems.

	The identity \eqref{eq:comm.identity.1} also holds in the $\CC^m$-valued setting with the same choice of $Q$, using the natural inner product on $\CC^m$-valued functions. Every operator appearing in \eqref{eq:comm.identity.1} then has scalar principal symbol, and the remainder of the proof from the scalar case can be followed without modification.
\end{proof}

The remaining case of $\gamma\in \WF_b(u)\setminus (\WF_b(f))\cup \WF(u_0)\cup \mathscr{G}_d\cup \mathscr{H})$ is treated in the scalar setting in \cite[Section~24.5]{Hormander3}, and this proof similarly carries over to the bundle setting by replacing pseudodifferential operators in the proof with the scalar matrices of such operators.

\section{Observability estimate}
\label{sec:observability}
In this section, we prove the following observability estimate for the connection wave equation \eqref{eq:wave.eq.2}. 
\begin{prop}
	\label{prop:observability}
	Let $(M,g)$ be of the form \eqref{eq:productform},\eqref{eq:metricform}, and let $E$ be a rank $N$ Hermitian vector bundle over $M$ equipped with a compatible connection. Suppose further that $g$ satisfies hypotheses (H2)-(H4) of Theorem \ref{thm:main} and let $\tilde{g}$ be a smooth Lorentzian metric on $M$ that lies in a sufficiently small $\SC^2(M)$-neighbourhood of $g$.

	Let $P=\Box+L\in \mathrm{Diff}^2(M;E)$ where $\Box$ is the connection wave operator associated to $(\tilde g,\nabla)$ and $L\in \mathrm{Diff}^1(M;E)$ has the property that $\sigma_1(L-L^*)\in S^1(T^*M;\End(E))$ is scalar.

	Then for $\delta >0$ sufficiently small and $\chi\in\SC_c^\infty((-T,T_0)\times \pa M_0)$ with $\chi|_{(-T+\delta,T_0-\delta)\times \pa M_0}=1$, and any $s\geq 0$, we have 
\begin{equation}
	\|\phi\|_{H_0^{s+1}(M_0;E)\times H_0^{s}(M_0;E)}\lesssim \|\chi \nabla_\nu u\|_{H_0^s(\Sigma;E)}
\end{equation}
where $u=\mathcal{S}(0,0,\phi)$ is the unique solution to \eqref{eq:wave.eq.2} with source $F=0$, boundary data $f=0$, final data $(u,\nabla_\nu u)|_{t=T}=\phi$ and $\Gamma=(-T+\delta,T_0-\delta)\times \pa M_0$.
\end{prop}
	The scalar version of this result was shown in \cite[Lemma~6.4]{LO1} to follow from \cite[Theorem~3.3]{BLR} (under the same geometric hypotheses (H2)-(H4)), with the essential point being that our geometric hypotheses (H2)-(H4) imply the geometric control condition for the set $\Gamma=(-T+\delta,T_0-\delta)$, that is, every bicharacteristic of $P$ passes through a nondiffractive point in $T^*\Gamma$.
	The details of this argument are unchanged in the bundle setting, and thus it suffices to prove the following bundle-valued version of \cite[Theorem~3.3]{BLR}.
\begin{prop}
	\label{prop:BLRobs}
Let $(M,g)$ be of the form \eqref{eq:productform},\eqref{eq:metricform}, and let $E$ be a rank $N$ Hermitian vector bundle over $M$ equipped with a compatible connection $\nabla$. 

	Let $$P=\Box+L\in \mathrm{Diff}^2(M;E)$$ where $\Box$ is the connection wave operator associated to $( g,\nabla)$ and $L\in \mathrm{Diff}^1(M;E)$ has the property that $\sigma_1(L-L^*)\in S^1(T^*M;\End(E))$ is scalar.
	Suppose that bicharacteristics of $P$ have finite order contact with $\partial(T^*M)$ and $\Gamma\subset (-T,T)\times M_0$ has the property that every compressed generalised bicharacteristic of $P$ passes through a nondiffractive point in $T^*\Gamma $.
	
	Then there is an $\ep >0$ such that for any $s\geq 1$ and any $u\in H^{s-1}(M;E)$ solving 
	\begin{align}
	P u&=0 \textrm{ in }(-T,T)\times M_0\\
	u&= 0 \textrm{ on $\Sigma=(-T,T)\times \pd M_0$}\label{eq:wave.eq.3.5}
\end{align}
	with $\nabla u|_{\Gamma}\in H^{s-1}(\Gamma;E\otimes T^*M)$, we have $u\in H^s((T-\ep,T)\times M_0;E)$ with the estimate 
	\begin{align}
		\label{eq:closed.graph.est}
	& c_1\|u\|_{H^s((T-\ep,T)\times M_0;E)}	\\
	& \leq \|\nabla_\nu u\|_{H^{s-1}(\Gamma;E)}+\|u\|_{H^s(\Gamma;E)}+c_2 \|u\|_{H^{s-1}((-T,T)\times M_0;E)}
	\end{align}
	for some $c_1 >0, c_2\geq 0$.
\end{prop}

\begin{proof}
We can follow the proof of \cite[Theorem~3.3]{BLR} in \cite{BLR}, replacing the Melrose--Sj\"{o}strand propagation of singularities result with the bundle-valued analogue Corollary \ref{cor:propagation} and \cite[Theorem~2.2]{BLR} with its bundle-valued analogue, which is remarked upon at the end of that proof.

For points of $T^*([-T,T]\times \pa M_0)$, we make use of the classification of boundary covectors as elliptic, hyperbolic, or glancing, as in \eqref{eq:boundary.classification}.

To show $u\in H^s((T-\ep,T)\times M_0;E)$, it is enough to show that the hypotheses imply that for any $p\in T_b^*M\setminus 0$ on the timeslice $\{T\}\times M_0$, $u$ is microlocally $H^s$ near $p$ in the sense that there exists a small conic neighbourhood $V \subseteq T_b^*M\setminus 0 $ of $p$ such that $\WF_b(Au)=\emptyset$ for all $A\in \Psi_b^0(M;E)$ with full symbol essentially supported in $V$.

For $p\in \mathrm{ell}(P)\subset T^*(\{T\}\times \mathrm{int}(M_0))$, this claim follows from the microlocal elliptic regularity theorem \cite[Theorem~18.1.28]{Hormander3}.
For $p\in \mathscr{E}$, $u$ is microlocally $H^s$ near $p$ from the microlocal elliptic boundary regularity theorem Proposition \ref{prop:elliptic.case}.

For $p\in \mathrm{char}(P)\cup \SCH \cup \SCG$ our assumptions imply that the unique compressed generalised bicharacteristic $\tilde\gamma$ through $p$ meets a non-diffractive point $q_1\in T^*\Gamma$ at some time $T_1 \in (-T,T)$.

By the bundle-valued analogue of \cite[Theorem~2.2]{BLR}, we have that $u$ is microlocally $H^s$ near $q_1$. Fix local coordinates $(x,y)$ for $M$ near $q_1$, where $x$ is a boundary defining function for $\pa M$, and fix a local trivialisation for $E$ by an orthonormal frame. We may assume that $q_1$ lies above the origin in the induced coordinate chart for $T_b^*M$. Let $a(y,\eta)\in S^0(T^*\RR^{n})$ be supported and equal to $1$ in a small conic neighbourhood of the covector corresponding to $q_1$.

Fixing $\phi\in\SC_c^\infty(\RR)$ supported and equal to $1$ near $0$, we consider the scalar matrix of tangential pseudodifferential operators $A(x,y,D_y)=\phi(x)a(y,D_y)\mathrm{Id}$, which has total symbol supported and equal to the $N\times N$ identity matrix in a conic neighbourhood of $T_b^*((T_1-\delta,T_1+\delta)\times M)\cap\tilde\gamma$ for sufficiently small $\delta >0$, in our  choice of local coordinates and local trivialisation.

Taking $\chi\in\SC^\infty(\RR)$ with $\spt(\chi)\subseteq \{t\geq T_1-\delta \}$ and $\chi|_{t\geq T_1}=1$, we consider the solution $U:=\chi u$ to 
\begin{align}
	P U&=F\\
	u&= 0 \textrm{ on $\Sigma=(-T,T)\times \pd M_0$}\label{eq:wave.eq.4}\\
	(u,\nabla_t u)&=0 \textrm{ on $\{T_1-\delta\}\times M_0$}
\end{align}
with source $F:=P(\chi u)=[P,\chi]u$. Writing $F=AF+(1-A)F=:F_1+F_2$ we can decompose $U=U_1+U_2$ as the sum of the solutions $U_j$ to \eqref{eq:wave.eq.4} with sources $F_j$. From the construction of $A$, we have that $\WF_b(f_2)\cap \tilde \gamma=\emptyset$, so by Corollary \ref{cor:propagation} we can conclude that $\WF_b(U_2)\cap \tilde\gamma=\emptyset$.

We exploit the microlocal regularity of $u$ near $q_1$ by writing
\begin{equation}
	F_1=AF=AP(\chi u)=A[P,\chi]u=[P,\chi]Au+[A,[P,\chi]]u.
\end{equation} 
The first term is a first order differential operator applied to $Au\in H^s$, and the second term is a zero-th order tangential pseudodifferential operator with total symbol essentially supported in a neighbourhood where $u$ has microlocal $H^s$ regularity. Hence $F_1\in H^{s-1}(M;E)$ and so an application of Proposition \ref{prop:forwardsmooth} (see Remark \ref{rem:general.source}) implies $U_1\in H^s(M;E)$.

We conclude that $u$ is microlocally $H^s$ near $p$, and an application of the closed graph theorem yields the estimate \eqref{eq:closed.graph.est}.
\end{proof}

\section{Gaussian beam construction}
\label{sec:goptics}
In this section, we carry out the construction of Gaussian beam solutions to the connection wave operator on Lorentzian manifolds. The classical construction of Gaussian beams for the scalar wave equation goes back to \cite{ralston.gaussian}. The generalisation of this construction to the connection wave operator on higher rank bundles is straightforward, however we shall present it here for the convenience of the reader. We follow the treatment in \cite{zeroth.order.wave}, where the construction of Gaussian beams is carried out in the setting of the scalar wave operator on a general Lorentzian manifolds. The proof of Lemma 8.1 can also be found in \cite[Lemma~2.4]{spyros.lauri.miika}.

Our setting is a rank $N$ Hermitian vector bundle $E$ over a Lorentzian manifold $(M,g)$ of the form \eqref{eq:productform},\eqref{eq:metricform}, equipped with a connection $\nabla$ satisfying \eqref{eq:conn.compat}. Further fixing a potential $V\in\SC^\infty(M;E) $ satisfying \eqref{eq:potential.compat}, we construct approximate solutions to 
\begin{equation}
	(\Box+V) u =0
\end{equation}
that concentrate on an inextendable null geodesic $\gamma$ of $M$, where $\Box$ is defined in \eqref{eq:box.def}, acting on $u\in\SC^\infty(M;E)$.

We begin by changing variables to Fermi coordinates around an interior segment of a null-bicharacteristic in $M$. These are a coordinate chart $(s=y_0,y_1,\ldots,y_n)$ in a tube $\ST=(s,y')\in (a,b)\times B(0,\delta)$ such that the coordinate representation of $\gamma$ is $(a,b)\times \{0\}$ and the metric tensor $g=\sum_{i,j=0}^{n} g_{ij}dy_i\otimes dy_j$ written in these coordinates satisfies
\begin{equation}
	\label{eq:fermi}
	g|_\gamma=2\, ds\otimes dy_1 + \sum_{j=2}^n dy_j \otimes dy_j 
\end{equation}
and $\pa_{y_k}(g^{ij})|_\gamma=0$ for each $k\in\{1,\ldots,n\}$. We refer the reader to \cite[Section~4.1,Lemma~1]{zeroth.order.wave} for a detailed construction of these coordinates in our setting. Abusing notation slightly, we shall also denote the pullback bundle over $\ST$ by $E$, and the pullback connection by $\nabla$.
We then take the Ansatz 
\begin{equation}\label{eq:ansatz}
	v_\lambda(s,y')=\chi(y')e^{i\lambda \phi(s,y')} a(s,y')
\end{equation}
where $\chi\in\SC_c^\infty(\RR^n)$ is equal to $1$ near zero and equal to zero outside $B(0,\delta/2)$.
The phase $\phi\in \SC^\infty(\ST)$ is assumed to be of the form
\begin{equation}\label{eq:formphi}
	\phi(s,y')=\sum_{j=0}^J \phi_j(s,y')
\end{equation}
with $\phi_j$ a homogeneous polynomial of degree $j$ in $y'$ with coefficients smooth in $s$, and the amplitude $a\in \SC^\infty(\ST;E)$ is assumed to be of the form 
\begin{equation}\label{eq:forma}
	a(s,y')=\sum_{k=0}^J  a_{k}(s,y')\lambda^{-k}
\end{equation}
 with 
 \begin{equation}
	a_k(s,y')=\sum_{j=0}^J a_{k,j}(s,y')
 \end{equation}
 where each $a_{k,j}$ is a homogeneous $E$-valued polynomial of degree $j$ in $y'$ with coefficients smooth in $s$. Using \eqref{eq:products}, we have
\begin{align}
	(\Box+V)(e^{i\lambda\phi}a)&=e^{i\lambda\phi}(\Box+V)a+\Box(e^{i\lambda \phi})a-2C(d(e^{i\lambda\phi}),\nabla a)\label{eq:P.ansatz}\\
	&= e^{i\lambda\phi}(\lambda^2\ang{d\phi,d\phi}_g-2i\lambda C(d\phi,\nabla a)+(Pa+a\Box \phi))
\end{align}
where $C$ denotes contraction of $T^*\ST \otimes T^*\ST$ using the metric tensor. The jets $\phi_j$ and $a_{k,j}$ along $\gamma$ are then determined inductively from \eqref{eq:P.ansatz} by asserting that the coefficients of each power of $\lambda$ in \eqref{eq:P.ansatz} vanish to order $J$ on $\gamma$.

The leading order term gives rise to the eikonal equation
\begin{equation}
	\label{eq:eikonal}
	\pa_{y'}^\alpha (\ang{d\phi,d\phi}_g)|_\gamma=0
\end{equation}
for $|\alpha|\leq J$. This is identical to the scalar case, and the construction of a $\phi=\sum_{j=0}^J \phi_j\in\mathcal{C}^\infty(\mathcal{T})$ satisfying \eqref{eq:eikonal} is carried out in \cite{zeroth.order.wave}.
The solution $\phi$ satisfies the three properties:
\begin{enumerate}
	\item $\mathrm{Im}( \phi ) \geq C|y'|^2\quad (C > 0)$;
	\item $\phi|_\gamma=0$;
	\item $\pa_{y_i} \phi|_\gamma=\delta_{1i}$.
\end{enumerate}

The first order term in $\lambda$ in \eqref{eq:P.ansatz} then gives rise to the transport equation 
\begin{equation}
	\label{eq:transport.1}
	\nabla_{y'}^\alpha (C(d\phi,\nabla a_0))|_\gamma=0
\end{equation}
for $|\alpha|\leq J$ and the lower order terms in $\lambda$ in \eqref{eq:P.ansatz} give rise to the transport equations
\begin{equation}
	\label{eq:transport.2}
	\nabla^\alpha_{y'}(2iC(d\phi,\nabla a_{k+1})+(P+\Box \phi)a_k)|_\gamma=0
\end{equation}
for $|\alpha| \leq J$. We shall see that \eqref{eq:transport.1} determines $a_0$, and then that \eqref{eq:transport.2} determines each subsequent $a_j$.

Taking $\alpha=0$ in \eqref{eq:transport.1}, we have the equation
\begin{equation}
	0=\sum_{i,j=0}^ng^{ij} (\pa_{y_i} \phi)(\nabla_{y_j} a_0) = \nabla_s a_0=\nabla_s a_{0,0}
\end{equation}
along $\gamma$, and so $a_{0,0}(s)$ satisfies a parallel transport equation. In particular, we can uniquely solve this equation for $s\in(a,b)$ with an arbitrary initial condition $a_{0,0}(s_0)=w\in E|_{\gamma(s_0)}$ where $s_0\in (a,b)$.

We proceed inductively by supposing that we have solved \eqref{eq:transport.1} up to order $m< J$ by choosing $a_{0,j}\in\mathcal{C}^\infty(\mathcal{T})$ appropriately for $j\leq m$, and showing that the equations \eqref{eq:transport.1} with $|\alpha|=m+1$ determine $a_{0,m+1}$ uniquely subject to the initial condition that $\nabla_{y'}^\alpha a_{0,m+1}(s_0,0)=0$ for all $|\alpha|=m+1$. 

To see this, we compute for $|\alpha|=m+1$ that
\begin{align}
	& \nabla_{y'}^\alpha \left(\sum_{i,j=0}^n g^{ij} (\pa_{y_i} \phi)(\nabla_{y_j} a_0)\right)\label{eq:transport.vanish}\\
	&=\sum_{\alpha_1+\alpha_2=\alpha}\sum_{i,j=0}^n c_{\alpha_1,\alpha_2 }\pa_{y'}^{\alpha_1}(g^{ij} \pa_{y_i} \phi)\nabla_{y'}^{\alpha_2}(\nabla_{y_j} a_0).
\end{align}
and consider the restriction to $\gamma$. The only terms involving at least $m+1$ derivatives landing on $a_0$ arise when $|\alpha_2|=m+1$ or $|\alpha_2|=m$. The term with $|\alpha_2|=m+1$ is of the form 
\begin{equation}
	\sum_{i,j=0}^n g^{ij}\pa_{y_i} \phi \nabla_{y'}^{\alpha}(\nabla_{y_j} a_0)|_\gamma =2(\nabla_{s}\nabla_{y'}^\alpha a_0+b_\alpha)|_\gamma 
\end{equation}
from condition (2), where $b_\alpha\in\SC^\infty(\gamma;E)$ is determined by $\nabla_{y'}^\beta a_0|_\gamma$ for $|\beta| \leq m$, $\phi$ and the coefficients of the connection $\nabla$. As such this $b_\alpha$ is already known.

Hence, \eqref{eq:transport.vanish} reduces to
\begin{equation}\label{eq:gaussian.system}
	2\nabla_s\nabla_{y'}^\alpha a_0+\sum_{|\beta|=m+1} A_\beta \nabla_{y'}^\beta a_0+b_\alpha
\end{equation}
along $\gamma$, where $b_\alpha\in\SC^\infty(\gamma;E)$ has been redefined and is determined by $\nabla_{y'}^\beta a_0|_\gamma$ for $|\beta| \leq m$, $\phi$, $g$, and the coefficients of the connection $\nabla$. The $A_\beta\in \SC^\infty(\gamma;\End(E))$ are also determined by $\phi,g,$ and the coefficients of the connection $\nabla$.

Using \eqref{eq:gaussian.system}, the equation $\eqref{eq:transport.1}$ then becomes a linear system of first order ODEs for $\{\nabla_{y'}^\alpha a_0(s,0):|\alpha|=m+1\}$ with smooth coefficients and smooth inhomogeneity, and so can be solved forward and backwards in time over the full interval $(a,b)$, subject to the initial conditions $\nabla_{y'}^\alpha a_0(s_0,0)=0$. This determines $a_{0,m+1}$ as claimed, and gives an inductive construction of $a_0$ by solving for $a_{0,j}$ one at a time.

The determination of $a_k$ for $k\geq 1$ from \eqref{eq:transport.2} is carried out inductively in $k$ and is essentially identical to the determination of $a_0$ from \eqref{eq:transport.1}. Indeed the additional term $(P+\Box \phi)a_k$ in \eqref{eq:transport.2} only contributes smooth inhomogeneities to the ODEs, which does not affect solvability. We omit the details for brevity.

\begin{lem}
	\label{lem:gaussian.beam}
	Suppose $\gamma:(a,b)\to M$ is an interior null geodesic segment as above and $\Omega\subset M$ is a subdomain with $\gamma(a),\gamma(b)\notin \ol\Omega$.

	Then for sufficiently small $\delta >0$, the function $v_\lambda$ defined in \eqref{eq:ansatz} with phase $\phi$ and amplitude $a$ constructed above (with any choice of initial condition $a(s_0,0)$), smoothly extended by $0$ in $\Omega$, satisfies the estimate
	\begin{equation}
		\|Pv_\lambda \|_{H^k(\Omega;E)}\leq C \lambda^{-K}
	\end{equation}
for $K=\frac{J+1}{2}+\frac{n}{4}-k-2$.
\end{lem}
\begin{proof}
	
By construction, $a$ and $\phi$ satisfy \eqref{eq:eikonal},\eqref{eq:transport.1} and \eqref{eq:transport.2}, so we can carry out the proof in \cite[Lemma~2]{zeroth.order.wave} without modification. The condition $\gamma(a),\gamma(b)\notin \Omega$ ensures that for sufficiently small $\delta > 0$, the ends of the tube $\{a\}\times B(0,\delta)$ and $\{b\}\times B(0,\delta)$ lies outside $\Omega$, and so $v_\lambda|_{\ST}$ can indeed by smoothly extended by $0$ in $\Omega$.
\end{proof}

We are now ready to prove Lemma \ref{lem:goptics}. 
\begin{proof}[Proof of Lemma \ref{lem:goptics}]
We begin by smoothly extending $M_0$ to a larger manifold with boundary $\tilde{M}_0$, and smoothly extending $g,\nabla$ to$\tilde{M}_0\times [-T,T]$ such that the metric form \eqref{eq:metricform} holds for the extension.

By linearity, it suffices to treat the case where $c_1=1$ and $c_j=0$ for $j > 1$. Fixing a choice of $\omega\in L^+_p M$, we consider the a segment of the null geodesic through $p$ with momentum $\omega$. That is, we consider the projection $\gamma$ of a null bicharacteristic $\tilde\gamma:(a,b)\to T^*(\mathrm{int}(\tilde{M}_0))\times [-T,T]$ with $\tilde\gamma(s_0)=(p,\omega)$ for some $s_0\in (a,b)$. We also assume that the segment is large enough so that  $\gamma(a),\gamma(b)\in (\tilde{M}_0\setminus M_0)\times [-T,T] $. This is possible from assumptions (H2),(H3) of Theorem \ref{thm:main}, together with \eqref{eq:D.defn}.

Using Lemma \ref{lem:gaussian.beam} with $\Omega=M_0\times [-T,T]$, we let $v_\lambda$ be a Gaussian beam associated to $\gamma$, with $M$ (and hence $K$) large, and with the initial conditions $a_k(s_0,0)=\delta_{1k}$ imposed.

Fix a smooth function $\eta(t)$ that is increasing, and equal to $0$ on $[-T,T_1+\ep/2]$, and equal to $1$ on $[T_1+\ep,T]$, where $\ep>0$ is sufficiently small so that $\gamma \cap (\pa M_0\times [T_1-\ep,T_1+\ep])=\emptyset $.

Then defining $u_\lambda $ to be the solution to \eqref{eq:wave.eq} with boundary data $\eta v_\lambda |_{[-T,T_1 +\ep ]\times \pa M_0 } $ using Proposition \ref{prop:forwardsmooth}, we have 
$P(u_\lambda-v_\lambda)=-Pv_\lambda $. Hence by the continuous dependence on source term in Proposition \ref{prop:forwardsmooth}, there follows the estimate
\begin{equation}\label{eq:beam.est}
	\|u_\lambda -v_\lambda \|_{C^1([-T,T_1+\ep/2]\times M_0;E)} \leq C/\lambda
\end{equation}
with constant $C$ uniform in $\lambda$.

To conclude the proof, we note that $p$ corresponds to the point $(s_0,0)$ in the Fermi coordinate chart, and use \eqref{eq:ansatz} to compute
\begin{equation}\label{eq:beam.value}
	v_\lambda(s_0,0)=a(s_0,0)
\end{equation}
and
\begin{align}
	\nabla v_\lambda(s_0,0)&=i\lambda a(s_0,0)\otimes d\phi(s_0,0)+\nabla a(s_0,0)\\
	&=i\lambda a(s_0,0)\otimes dy_1 + \nabla a(s_0,0).\label{eq:beam.derivative}
\end{align}
Near $p$ along $\gamma$, we have $\dot\gamma(s)=\pa_s$ in Fermi coordinates and the Hamiltonian vector field is given by $2\eta_1\pa_s +2\xi\pa_{y_1}+2\sum_{j=2}^n \eta_j \pa_{y_j} $, where $(\zeta,\eta)$ are the cotangent coordinates dual to $(s,y)$. Hence we have $\tilde\gamma(s_0)=(p,\frac12 dy_1)$, that is $2\omega=dy_1$.

In our construction of $a$ following \eqref{eq:transport.2}, it was also shown that we could take arbitrary initial condition $a(s_0,0)$, in particular we can take $a(s_0,0)=E_1$ in terms of the local orthonormal frame for $E$ near $p$.
Inserting this information into \eqref{eq:beam.value}, \eqref{eq:beam.derivative}, and using \eqref{eq:beam.est} completes the proof.
\end{proof}

\section{Energy estimates}
\label{sec:direct}
In this section we derive the energy estimates that are used in the proof of well-posedness for the direct problem, that is Proposition \ref{prop:forwardsmooth}. These computations are analogous to the treatment of the scalar case in \cite{llt}, but are adapted to the present setting of the connection wave operator on a fixed trivial Hermitian vector bundle.
Let $$M=[0,T]\times M_0 $$
be a $(1+n)$ dimensional Lorentzian manifold with metric \eqref{eq:metricform}, and let $E$ be a rank $N$ Hermitian vector bundle over $M$ equipped with a compatible connection $\nabla$. Let $$P=\Box+L\in\mathrm{Diff}^2(M;E) $$where $L\in\mathrm{Diff}^1(M;E)$ is arbitrary and $\Box$ is the connection wave operator \eqref{eq:box.def.0} associated to $(g,\nabla)$.

We shall consider smooth solutions $u\in \SC^\infty(M;E)$ to the following connection wave equation
\begin{align}
	P u&=F\\
	u&= f \textrm{ on $(0,T)\times \pd M_0$}\label{eq:wave.eq.3}\\
	(u,\nabla_t u)&=\phi \textrm{ on $\{0\}\times M_0$.}
\end{align}
\begin{rem}
The equation \eqref{eq:wave.eq.3} differs from \eqref{eq:wave.eq.2} only by a reflection and translation in $t$, which has no impact on well-posedness.
\end{rem}
Introducing the notation $v(t)$ for the restriction of arbitrary $v\in \SC^\infty(M;E)$ to time $t$, we define the energy
\begin{equation}
	E(t):=\frac{1}{2}(\|u(t)\|_{L^2(M_0;E)}^2)+\||g^{00}|^{1/2}\nabla_t u(t)\|_{L^2(M_0;E)}^2+ \|\nabla_x u(t) \|_{L^2(M_0;E\otimes T^*M_0)}^2 ),
\end{equation}
where $g^{00}=\langle  dt,dt\rangle_g $, $\nabla_x$ denotes the (time-dependent) restriction of the connection $\nabla$ to $\pi^{-1}(\{t\}\times M_0)\subset E$, and $L^2(M_0;E\otimes T^*M_0)$ is defined by the bilinear pairing \eqref{eq:L2.def.forms} with $M$ replaced by $M_0$. We use the notation $dV$ and $dS$ to denote the volume densities on $M_0$ and $\pa M_0$ induced by $g$, suppressing the $t$-dependence of these forms, as well as the corresponding Sobolev space norms. 
The basic energy estimate is as follows.
\begin{prop}
\label{prop:energy.base}
	\begin{equation}
		E(t)\lesssim \|Pu\|^2_{L^2([0,t]\times M_0;E)}+ \|\nabla_t u\|_{L^2([0,t]\times \pa M_0 ;E)}^2+\|u\|^2_{L^2([0,t],H^1(\pa M_0;E))}+E(0)
	\end{equation}
	For arbitrary $X\in\SC^\infty(M;TM)$, we have 
	\begin{align}
		\|\nabla_X u\|^2_{L^2([0,t]\times \pa M_0;E)}&\lesssim \|Pu\|_{L^2([0,t]\times M_0;E)}^2+\|\nabla_t u\|^2_{L^2([0,t]\times\pa M_0;E)}\\
		&+\|u\|^2_{L^2([0,t],H^1(\pa M_0;E))}+E(0)
	\end{align}
\end{prop}
\begin{proof}
Differentiating in $t$, we compute  
\begin{equation}
	\pa_t (\|u(t)\|_{L^2(M_0;E)}^2)\lesssim E(t),
\end{equation}
\begin{align}
	\pa_t(\||g^{00}|^{1/2}\nabla_t u(t)\|_{L^2(M_0;E)}^2)&=2\Re \left(-\int_{M_0} \ang{\nabla_t^2 u(t),\nabla_t u(t)}_E g^{00}\, dV\right)\\&+O(E(t)),
\end{align}
and 
\begin{align}
	\pa_t (\|\nabla_x u(t) \|_{L^2(M_0;E)}^2)&= 2\Re \left(\int_{M_0}\ang{ \nabla_x u(t),\nabla_x\nabla_t  u(t)}_{E\otimes T^*M_0} \, dV \right)\\&+ O(E(t))\\
	&= 2\Re \Big(\int_{M_0} \ang{\nabla_x^*\nabla_x u(t),\nabla_t u(t)}_E dV\\&+\int_{\pa M_0}\ang{\nabla_\nu u(t),\nabla_t u(t)}_E\, dS\Big)+O(E(t)).\\
\end{align}
From \eqref{eq:tensor.laplacian}, it follows that $P$ is equal to $-g^{00}\nabla_t^2+\nabla_x^*\nabla_x$ up to terms of order at most $1$.
Hence we obtain 
\begin{align}
	\partial_tE(t)&\lesssim  \|Pu(t)\|_{L^2(M_0;E)}^2 + \ep\|\nabla_\nu u(t)\|^2_{L^2(\pa M_0;E)}\\ &+\ep^{-1}\|\nabla_t u(t)\|_{L^2(\pa M_0;E)}^2+ E(t).
\end{align}
Integrating this estimate over $[0,t]$ gives
\begin{align}
	E(t)&\lesssim E(0)+\|Pu\|_{L^2([0,t]\times M_0;E)}^2+\ep\|\nabla_\nu u\|^2_{L^2([0,t]\times\pa M_0;E)}\\ &+\ep^{-1}\|\nabla_t u\|_{L^2([0,t]\times\pa M_0;E)}^2  +\int_0^t E(s)\, ds
\end{align}
and an application of Gr\"{o}nwall's inequality gives
\begin{align}\label{eq:energy.1}
	E(t)&\lesssim E(0)+\|Pu\|_{L^2([0,t]\times M_0;E)}^2+\ep\|\nabla_\nu u\|^2_{L^2([0,t]\times\pa M_0;E)}\\&+\ep^{-1}\|\nabla_t u\|_{L^2([0,t]\times\pa M_0;E)}^2.
\end{align}
Next we bound the $\nabla_\nu u$ term in \eqref{eq:energy.1}. Taking $X\in\SC^\infty(M;TM)$ with $X|_\Sigma = \nu$, we have
\begin{align}
	&  \int_{M_0} \ang{\nabla_x^*\nabla_x u(t),\nabla_X u(t)}_E \, dV\\
	&= \int_{M_0} \ang{\nabla_x u(t),\nabla_x \nabla_X u(t)}_{E\otimes T^*M_0}\,dV\\ &-\int_{\pa M_0 }\ang{\nabla_x u(t),\nabla_X u(t) \otimes \nu^*}_{E\otimes T^* M_0}\, dS\\
	&= \frac{1}{2}\int_{M_0} X\|\nabla_x u(t)\|_{E\otimes T^* M}^2\, dV-\frac12 \int_{\pa M_0}\|\nabla_\nu u(t)\|_E^2\, dS+O(E(t))\\
	&= -\frac{1}{2}\int_{M_0} \|\nabla_x u(t)\|_{E\otimes T^* M_0}^2 \mathrm{div}(X)\, dV\\ &+\frac{1}{2}\int_{\pa M_0} \|\nabla_x u(t)\|_{E\otimes T^*M_0}^2 -\|\nabla_\nu u(t)\|_E^2\, dS+O(E(t))\\
	&= \frac{1}{2}(\|u(t)\|_{H^1(\pa M_0;E)}^2-\|\nabla_\nu u(t)\|^2_{L^2(M_0;E)})+O(E(t))\label{eq:energy.boundary.1}
\end{align}
where we have used that $$\|\nabla_x u(t)\|^2_{L^2(\pa M_0;E)}=\|\nabla_\nu u(t)\|^2_{L^2(\pa M_0;E)}+\|u(t)\|^2_{H^1(\pa M_0;E)}.$$
Similarly, we compute 
\begin{align}
	& \int_0^t \int_{M_0} \ang{\nabla_t^2 u(s),\nabla_X u(s)}_Eg^{00}\, dV\, ds\\
	&= \int_0^t\int_{M_0} \pa_t(\ang{\nabla_t u(s),\nabla_X u(s)}_E g^{00})\, dV\, ds \\&- \int_0^t\int_{M_0}\ang{\nabla_t u(s),\nabla_X \nabla_t u(s) }_Eg^{00}\, dV \, ds +O(E(t))\\
	&= \left[\int_{M_0} \ang{\nabla_t u(s),\nabla_X u(s)}_E g^{00}\, dV\right]_{s=0}^{s=t}\\&+\frac{1}{2}\int_0^t \int_{M_0}\|\nabla_t u(s)\|_E^2 \mathrm{div}(X)g^{00}\, dV\, ds\\
	&+\frac12\||g^{00}|^{1/2}\nabla_t u\|^2_{L^2([0,t]\times \pa M_0;E)}+\int_0^t O(E(s)) \, ds\\
	&= \left[\int_{M_0} \ang{\nabla_t u(s),\nabla_X u(s)}_E g^{00}\, dV\right]_{s=0}^{s=t}\\&+\frac12\||g^{00}|^{1/2}\nabla_t u(s)\|^2_{L^2([0,t]\times \pa M_0;E)}+\int_0^t O(E(s)) \, ds\label{eq:energy.boundary.2} .
\end{align}
We now integrate \eqref{eq:energy.boundary.1} over $s\in[0,t]$ and subtract \eqref{eq:energy.boundary.2}. Once more using that $P=-g^{00}\nabla_t^2+\nabla_x^*\nabla_x$ up to lower order terms, we obtain
\begin{align}
	& \int_0^t\int_{M_0} \ang{Pu(s),\nabla_X u(s)}_E \, dV\, ds\\
	&=-\left[\int_{M_0} \ang{\nabla_t u(s),\nabla_X u(s)}_E g^{00}\, dV\right]_{s=0}^{s=t}+\frac{1}{2}(\|u\|^2_{L^2([0,t],H^1(\pa M_0;E))}\\
	&-\|\nabla_\nu u\|^2_{L^2([0,t]\times \pa M_0;E)}-\||g^{00}|^{1/2}\nabla_t u\|_{L^2([0,t]\times \pa M_0;E )}^2+O(E(s)))
\end{align}
which gives 
\begin{align}
	& \|\nabla_\nu u\|^2_{L^2([0,t]\times \pa M_0;E)}+\||g^{00}|^{1/2}\nabla_t u\|_{L^2([0,t]\times \pa M_0;E)}^2\\
	&\lesssim \|Pu\|^2_{L^2([0,t]\times M_0;E)}+\|u\|^2_{L^2([0,t] ,H^1(\pa M_0;E))}\\ &+\int_0^t E(s)\, ds + E(t)+E(0). \label{eq:energy.2}
\end{align}
We now combine \eqref{eq:energy.1} and \eqref{eq:energy.2} to obtain
\begin{align}
	& E(t)\\
	&\leq E(0)+\|Pu\|^2_{L^2([0,t]\times M_0;E)}+ \ep^{-1}\|\nabla_t u\|_{L^2([0,t]\times \pa M_0;E)}^2 \\
	&+ \ep(\|Pu\|^2_{L^2([0,t]\times M_0;E)}+\|u\|^2_{L^2([0,t],H^1(\pa M_0;E))}\\&+\int_0^t E(s)\, ds + E(t)+E(0)).
\end{align}
Taking $\ep > 0$ small and absorbing the $\ep E(t)$ into the left-hand side, an application of Gr\"onwall's inequality then gives 
\begin{align}
	\label{eq:energy.0}
	E(t)&\lesssim \|Pu\|^2_{L^2([0,t]\times M_0;E)}+ \|\nabla_t u\|_{L^2([0,t]\times \pa M_0;E)}^2\\&+\|u\|^2_{L^2([0,t],H^1(\pa M_0;E))}+E(0).
\end{align}
Inserting this estimate back into \eqref{eq:energy.2} gives 
\begin{align}
	\|\nabla_\nu u\|^2_{L^2([0,t]\times \pa M_0;E)}&\lesssim \|Pu\|_{L^2([0,t]\times M_0;E)}^2+\|\nabla_t u\|^2_{L^2([0,t]\times\pa M_0;E)}\\&+\|u\|^2_{L^2([0,t],H^1(\pa M_0;E))}+E(0).
\end{align}
Furthermore, for arbitrary $X\in\SC^\infty(M;TM)$, we have the analogous estimate 
\begin{align}
	\label{eq:energy.deriv.0}
	\|\nabla_X u\|^2_{L^2([0,t]\times \pa M_0)}&\lesssim \|Pu\|_{L^2([0,t]\times M_0;E)}^2+\|\nabla_t u\|^2_{L^2([0,t]\times\pa M_0;E)}\\&+\|u\|^2_{L^2([0,t],H^1(\pa M_0;E))}+E(0)
\end{align}
as any tangential derivative of $u$ on $[0,t]\times \pa M_0$ is trivially controlled by the right-hand side.
\end{proof}
Next, we inductively extend the energy estimates \eqref{eq:energy.0},\eqref{eq:energy.deriv.0} to higher order Sobolev scales. To this end, we introduce the following notation.
\begin{equation}
	E_k(t)=E_k[u](t)=\|\nabla_t u\|^2_{H^k(M_0;E)}+\|u\|^2_{H^{k+1}(M_0;E)}
\end{equation}
\begin{equation}
	R_k(t)=R_k[f](t)=\sum_{j=0}^k \|\nabla_t^j f\|^2_{L^2([0,t],H^{k-j}(M_0;E))}
\end{equation}
\begin{equation}
	B_k(t)=B_k[u](t)=\sum_{j=0}^{k+1} \|\nabla_t^j u\|^2_{L^2([0,t],H^{k+1-j}(\pa M_0;E))}.
\end{equation}
The higher order energy estimate is as follows.
\begin{prop}
	If $u\in\SC^\infty(M;E)$ satisfies $Pu=F$, and $X\in\mathrm{Diff}^{k+1}(M;E)$ is arbitrary, then we have
	\begin{align}\label{eq:energy.main}
		E_k(t)+\|Xu\|^2_{L^2([0,t]\times \pa M_0;E)}\lesssim \int_0^t R_k[F](s)+B_k(s)\, ds + E_k(0).
	\end{align}
\end{prop}
\begin{proof}
	The $k=0$ case of this result is directly implied by Proposition \ref{prop:energy.base}. We proceed by induction, assuming the result is known up to order $k$. Throughout the proof, the notation $Q_j$ shall denote an arbitrary ($t$-dependent) element of $\mathrm{Diff}^j(M_0;E)$, possibly varying from line to line. 
	Let $X\in \mathrm{Diff}^{k+1}(M;TM)$, and let $Y\in\mathcal{C}^\infty(M;TM)$.
	Then an application of the order $k$ version of \eqref{eq:energy.main} with $u$ replaced by $\nabla_Y u$ gives
	\begin{align}
		E_k[\nabla_Yu](t)+\|X \nabla_Y u\|^2_{L^2(\Sigma;E)}&\lesssim \sum_{j=0}^k \|\nabla_t^j(P\nabla_Y u)\|^2_{L^2([0,t],H^{k-j}(M_0;E))}\\
		&+ \sum_{j=0}^{k+1}\|\nabla_t^j\nabla_Y u\|^2_{L^2([0,t],H^{k+1-j}(\pa M_0;E))}\\
		&+ E_k[\nabla_Y u](0).
	\end{align}
	In the first right-hand side term, we have $\nabla_t^j P \nabla_Y u=\nabla_Y\nabla_t^j P u+[\nabla_t^j P,\nabla_Y]u$. We can immediately bound 
	\begin{align}
		\label{eq:energy.induction.1}
		\|\nabla_Y\nabla_t^j Pu\|_{L^2([0,1],H^{k-j}(M_0;E))}&\lesssim \|\nabla_t^j F\|_{L^2([0,1],H^{k+1-j}(M_0;E))}\\&+\|\nabla_t^{j+1}F\|_{L^2([0,t],H^{k-j}(M_0;E))}.
	\end{align}
	On the other hand, $[\nabla_t^j P,\nabla_Y]\in \mathrm{Diff}^{j+2}(M;E)$. We can always write an element of $ \mathrm{Diff}^{j+2}(M)$ in the form $Q_0\nabla_t^{j+2}+\ldots+Q_{j+1}\nabla_t+Q_{j+2}$, and so we also have 
	\begin{equation}\label{eq:energy.order.red}
		\nabla_t^2 u =Q_0F + Q_1 \nabla_t u + Q_2 u
	\end{equation}
	using $Pu=F$.
	We can repeatedly use \eqref{eq:energy.order.red} to reduce powers of $\nabla_t$ acting upon $u$, until we arrive at 
	\begin{align}
		\label{eq:energy.order.est}
		\|[\nabla_t^j P,\nabla_Y] u\|_{L^2([0,t],H^{k-j}(M_0;E))}&\lesssim \sum_{m=0}^j \|\nabla_t^m F\|_{L^2([0,t],H^{k-m}(M_0;E))}\\
		&+\|u\|_{L^2([0,t]\times H^{k+2}(M_0;E))}\\
		&+\|\nabla_t u\|_{L^2([0,t],H^{k+1}(M_0;E))}.
	\end{align}
	Upon summing \eqref{eq:energy.induction.1} and \eqref{eq:energy.order.est} from $j=0$ to $k$ we arrive at
	\begin{equation}
		\|\nabla_t^j(P\nabla_Y u)\|^2_{L^2([0,t],H^{k-j}(M_0;E))}\lesssim \int_0^t R_{k+1}[F](s)+E_{k+1}(s)\, ds.
	\end{equation}
	Next, we bound $E_k[\nabla_Y u](0)$. We have the immediate estimate
	\begin{equation}
		\|\nabla_Y u(0)\|^2_{H^{k+1}(M_0;E)}\lesssim \|u(0)\|^2_{H^{k+2}(M_0;E)}+\|\nabla_t u(0)\|^2_{H^{k+1}(M_0;E)}
	\end{equation}
	and since $\nabla_t\nabla_Y$ is second order, we can use \eqref{eq:energy.order.red} to write
	\begin{align}
		\|\nabla_t \nabla_Y u(0)\|^2_{H^k(M_0;E)}&=\|(Q_0\nabla_t^2+Q_1\nabla_t+Q_2)u(0)\|^2_{H^k(M_0;E)}\\
		&=\|Q_0F(0)+(Q_1\nabla_t+Q_2)u(0)\|^2_{H^{k}(M_0;E)}\\
		&\leq \|F(0)\|^2_{H^k(M_0;E)}+ E_{k+1}[u](0).
	\end{align}
	We have now shown that 
	\begin{align}
		\label{eq:nowshown}
		& E_k[\nabla_Y u](t)+ \|X\nabla_Y u\|^2_{L^2([0,t]\times \pa M_0;E)}\\
		&\lesssim \int_0^s R_{k+1}(s)+B_k[\nabla_Y u](s)+E_{k+1}(s)\, ds +E_{k+1}(0).
	\end{align}
	For $Y$ tangential to $[0,t]\times \pa M_0$, we have $B_k[\nabla_Y u]\lesssim B_{k+1}[u]$, so taking $Z_1,\ldots,Z_{k+1}$ tangential to $[0,t]\times \pa M_0$, an application of \eqref{eq:nowshown} yields
	\begin{align}
		\|\nabla_{Z_1}\ldots \nabla_{Z_k}\nabla_Y \nabla_{Z_{k+1}}u\|^2_{L^2([0,t]\times \pa M_0;E)}&\lesssim \int_0^t R_{k+1}(s)+B_{k+1}(s)+E_{k+1}(s)\, ds \\&+E_{k+1}(0)
	\end{align}
	whether or not $Y$ is tangential. As the commutator $[\nabla_{Z_{k+1}},\nabla_Y]$ lies in  $\mathrm{Diff}^1(M;E)$, our inductive hypothesis implies 
	\begin{equation}
		\label{eq:nowshown2}
		\|Z\nabla_Y u\|^2_{L^2([0,t]\times \pa M_0;E)}\lesssim \int_0^t R_{k+1}(s)+B_{k+1}(s)+E_{k+1}(s)\, ds +E_{k+1}(0)
	\end{equation}
	for any tangential $Z\in\mathrm{Diff}^{k+1}(M;E)$. In particular, it follows that 
	\begin{equation}
		\int_0^t B_k[\nabla_Y u](s)\, ds\lesssim \int_0^t R_{k+1}(s)+B_{k+1}(s)+E_{k+1}(s)\, ds +E_{k+1}(0)
	\end{equation}
	and so \eqref{eq:nowshown2} holds without the assumption that $Y$ is tangential. Inserting this estimate into \eqref{eq:nowshown} yields 
	\begin{equation}
		\label{eq:nowshown3}
		\|X u\|^2_{L^2([0,t]\times \pa M_0;E)}\lesssim \int_0^t R_{k+1}(s)+B_k(s)+E_{k+1}(s)\, ds +E_{k+1}(0)
	\end{equation}
	for any $X\in \mathrm{Diff}^{k+2}(M;E)$.
	It remains to bound $E_{k+1}(t)$. Let $Y_1,\ldots Y_l\in \SC^\infty(M_0;TM_0)$ span $T_pM_0$ at every point $p\in M_0$, so in particular $$\|u(t)\|_{H^{m+1}(M_0;E)}\lesssim \sum_{j=1}^l\|\nabla_{Y_j}u(t)\|_{H^m(M_0;E)}$$ for any $m\in\mathbb{N}$.
	We can then bound 
	\begin{equation}
		\|u(t)\|_{H^{k+2}(M_0;E)}\lesssim \sum_{j=1}^l \|\nabla_{Y_j}u(t)\|_{H^{k+1}(M_0;E)}\lesssim \sum_{j=1}^l E_k[\nabla_{Y_j}u](t).
	\end{equation}
	and 
	\begin{align}
		\|\nabla_t u(t)\|_{H^{k+1}(M_0;E)}&\lesssim \sum_{j=1}^l \|\nabla_{Y_j}\nabla_t u \|_{H^k(M_0;E)}\\
		&\lesssim \sum_{j=1}^l \|\nabla_t\nabla_{Y_j} u(t) \|_{H^k(M_0;E)}\\&+\|[\nabla_{Y_j},\nabla_t]u(t)\|_{H^k(M_0;E)}\\
		&\lesssim \sum_{j=1}^l E_k[\nabla_{Y_j}u](t)+E_k(t)
	\end{align}
	as the commutator is first order and hence of the form $Q_0+Q_1D_t$.
	Hence it follows from \eqref{eq:nowshown} and \eqref{eq:nowshown3} that
	\begin{align}
		E_{k+1}(t)+\|Xu\|^2_{L^2([0,t]\times \pa M_0;E)}&\lesssim \int_0^t R_{k+1}[F](s)+B_{k+1}(s)+E_{k+1}(s)\, ds \\&+ E_{k+1}(0).
	\end{align}
	An application of Gr\"{o}nwall's inequality then completes the proof.
\end{proof}

\bibliography{all}{}
\bibliographystyle{plain}
\end{document}